\documentclass[11pt]{article}

\usepackage{amsxtra,amssymb,amsthm,amsmath,latexsym}
\usepackage{graphicx}
\usepackage{epsfig}
\usepackage{afterpage}
\newtheorem{thm}{Theorem}[section]

\newtheorem{lem}[thm]{Lemma}
 \newcommand{\thmref}[1]{Theorem~\ref{#1}}
 \newcommand{\lemref}[1]{Lemma~\ref{#1}}

\newcommand{\ol}{\overline}

\newcommand{\la}{{\langle}}
\newcommand{\ra}{{\rangle}}
\newcommand{\dl}{{\delta}}
\newcommand{\bee}{\begin{equation*}}
\newcommand{\eee}{\end{equation*}}
\newcommand{\be}{\begin{equation}}
\newcommand{\ee}{\end{equation}}
\newcommand{\pn}{\par\noindent}
\title{An iterative method for solving Fredholm integral equations of the first kind}
\author{Sapto W. Indratno \\
\small Department of Mathematics\\[-0.8ex]
\small Kansas State University, Manhattan, KS 66506-2602, USA\\
\small \texttt{sapto@math.ksu.edu}
\and
A.G. Ramm$^*$\\
\small Department of Mathematics\\[-0.8ex]
\small Kansas State University, Manhattan, KS 66506-2602, USA\\[-0.8ex]
\small \texttt{ramm@math.ksu.edu}\\
$^*$Corresponding author
}

\date{}
\begin{document}

\maketitle

\begin{abstract}
The purpose of this paper is to give a convergence analysis of the
iterative scheme: \bee
u_n^\dl=qu_{n-1}^\dl+(1-q)T_{a_n}^{-1}K^*f_\dl,\quad u_0^\dl=0,\eee
where $T:=K^*K,\quad T_a:=T+aI,\quad q\in(0,1),\quad
a_n:=\alpha_0q^n,\ \alpha_0>0,$ with finite-dimensional
approximations of $T$ and $K^*$ for solving stably Fredholm integral
equations of the first kind with noisy data.
\end{abstract}
\pn{

 }\\

\noindent\textbf{Biographical notes:} Professor Alexander G. Ramm is
an author of more than 580 papers, 2 patents, 12 monographs, an
editor of 3 books, and an associate editor of several mathematics
and computational mathematics Journals. He gave more than 135
addresses at various Conferences, visited many Universities in
Europe, Africa, America, Asia, and Australia. He won Khwarizmi Award
in Mathematics, was Mercator Professor, Distinguished Visiting
Professor supported by the Royal Academy of Engineering, invited
plenary speaker at the Seventh PanAfrican Congress of
Mathematicians, a London Mathematical Society speaker, distinguished
HKSTAM speaker, CNRS research professor, Fulbright professor in
Israel, distinguished Foreign professor in Mexico and Egypt. His
research interests include inverse and ill-posed problems,
scattering theory, wave propagation, mathematical physics,
differential and integral equations, functional analysis, nonlinear
analysis, theoretical numerical analysis, signal processing, applied
mathematics and operator theory.\\

\noindent Sapto W. Indratno is currently a PhD student at Kansas
State University under the supervision of Prof. Alexander G. Ramm.
He is a coauthor of three accepted papers. His fields of interest
are numerical analysis, optimization, stochastic processes, inverse
and ill-posed problems, scattering theory, differential equations
and applied mathematics.

\section{Introduction}
We consider a linear operator
\begin{equation}\label{11}
(Ku)(x):=\int_a^bk(x,z)u(z)dz=f(x),\quad a\leq x\leq b,
\end{equation}
where $K:L^2[a,b]\to L^2[a,b]$ is a linear compact operator. We
assume that $k(x,z)$ is a smooth function on $[a,b]\times[a,b]$.
Since $K$ is compact, the problem of solving equation \eqref{11} is
ill-posed. Some applications of the Fredholm integral equations of
the first kind can be found in \cite{VIVT02}, \cite{RAMM05},
\cite{RAMM499}. There are many methods for solving equation
\eqref{11}: variational regularization, quasi-solution, iterative
regularization, the Dynamical Systems Method (DSM). A detailed
description of these methods can be found in \cite{MRZ84},
\cite{RAMM05}, \cite{RAMM499}. In this paper we propose an iterative
scheme for solving equation \eqref{11} based on the DSM. We refer
the reader to \cite{RAMM05} and \cite{RAMM499} for a detailed
discussion of the DSM. When we are trying to solve \eqref{11}
numerically, we need to carry out all the computations with
finite-dimensional approximation $K_m$ of the operator $K$,
$\lim_{m\to \infty}\|K_m-K\|=0$. One approximates a solution to
\eqref{11} by a linear combination of basis functions
$v_m(x):=\sum_{i=1}^m \zeta_j^{(m)}\phi_j(x)$, where $\zeta_j^{(m)}$
are constants, and $\phi_i(x)$ are orthonormal basis functions in
$L^2[0,1]$. Here the constants $\zeta_j^{(m)}$ can be obtained by
solving the ill-conditioned linear algebraic system: \be\label{12}
\sum_{j=1}^m(K_m)_{ij}\zeta_j=g_i,\quad i=1,2,\hdots,m, \ee where
$(K_m)_{ij}:=\int_a^b\int_a^bk(x,s)\phi_j(s)ds\ol{\phi_i(x)}dx$,
$1\leq i,j\leq m$, and $g_i:=\int_a^b f(x)\ol{\phi_i(x)}dx$. In
applications, the exact data $f$ may not be available, but noisy
data $f_\dl$, $\|f_\dl-f\|\leq \dl$, are available. Therefore, one
needs a regularization method to solve stably equation \eqref{12}
with the noisy data $g_i^\dl:=\int_a^bf_\dl(x)\ol{\phi_i(x)}dx$ in
place of $g_i$. In the variational regularization (VR) method for a
fixed regularization parameter $a>0$ one obtains the coefficients
$\zeta_j^{(m)}$ by solving the linear algebraic system:
\be\label{13}
a\zeta_i^{(m)}+\sum_{j=1}^m(K_m^*K_m)_{ij}\zeta_j^{(m)}=g_i^\dl,\quad
i=1,2,\hdots,m, \ee where
$$(K^*_mK_m)_{ij}:=\int_a^b\int_a^b\overline{k(s,x)\phi_i(x)}\int_a^bk(s,z)\phi_j(z)dzdsdx,$$ $\|f-f_\dl\|\leq \dl$, and $\overline{k(s,x)}$ is the complex conjugate of $k(s,x)$.
In the VR method one has to choose the regularization parameter $a$.
In \cite{MRZ84} the Newton's method is used to obtain the parameter
$a$ which solves the following nonlinear equation: \be\label{nevr}
F(a):=\|K_m\zeta_m-g^\dl\|^2=(C\dl)^2,\quad C\geq 1, \ee where
$\zeta_m=(aI+K_m^*K_m)^{-1}K_m^*g^\dl,$ and $K^*_m$ is the adjoint
of the operator $K_m$. In \cite{SWIAGR09} the following iterative
scheme for obtaining the coefficients $\zeta_j^{(m)}$ is studied:
\be\label{it1a}
\zeta_{n,m}^\dl=q\zeta_{n-1,m}^\dl+(1-q)T_{a_n,m}^{-1}K_m^*g^\dl,\quad
d_0^\dl=0,\quad a_n:=\alpha_0q^n, \ee where $\alpha_0>0,$
$q\in(0,1),$  \be\label{Tm} T_{a,m}:=T^{(m)}+aI,\quad
T^{(m)}:=K_m^*K_m,\quad a>0, \ee and $I$ is the identity operator.
Iterative scheme \eqref{it1a} is derived from a DSM solution of
equation \eqref{11} obtained in \cite[p.44]{RAMM05}. In iterative
scheme \eqref{it1a} adaptive regularization parameters $a_n$ are
used. A discrepancy-type principle for DSM is used to define the
stopping rule for the iteration processes.

The value of the parameter $m$ in \eqref{nevr} and \eqref{it1a} is
fixed at each iteration, and is usually large. The method for
choosing the parameter $m$ has not been discussed in
\cite{SWIAGR09}. In this paper we choose the parameter $m$ as a
function of the regularization parameter $a_n$, and approximate the
operator $T:=K^*K$ (respectively $K^*$) by a finite-rank operator
$T^{(m)}$ (respectively $K_m^*$): \be\label{c1} \lim_{m\to
\infty}\|T^{(m)}-T\|=0. \ee Condition \eqref{c1} can be satisfied by
approximating the kernel $g(x,z)$ of $T$, \be\label{gxz}
g(x,z):=\int_a^b\overline{k(s,x)}k(s,z)ds ,\ee with the degenerate
kernel \be\label{gm}
g_m(x,z):=\sum_{i=1}^mw_i\overline{k(s_i,x)}k(s_i,z), \ee where
$\{s_i\}_{i=1}^m$ are the collocation points, and $w_i$,$1\leq i
\leq m,$ are the quadrature weights. Quadrature formulas \eqref{gm}
can be found in \cite{PDPR84}. Let $K_m^*$ be a finite-dimensional
approximation of $K^*$ such that \be\label{c2} \lim_{m\to
\infty}\|K^*-K_m^*\|=0. \ee One may choose $K_m^*=P_mK^*$, where
$P_m$ is a sequence of orthogonal projection operators on $L^2[a,b]$
such that $P_mx\to x\text{ as } m\to \infty,$ $\forall x\in
L^2[a,b]$. We propose the following iterative scheme: \be\label{it3}
u_{n,m_n}^\dl=qu_{n-1,m_{n-1}}^\dl+(1-q)T_{a_n,m_n}^{-1}K_{m_n}^*f_\dl,\quad
u_{0,m_0}^\dl=0, \ee where $a_n:=\alpha_0q^n,$ $\alpha_0>0,$
$q\in(0,1),$ $\|f_\dl-f\|\leq \dl$, $T_{a,m}$ is defined in
\eqref{Tm} with $T^{(m)}$ satisfying condition \eqref{c1}, $K_m^*$
is chosen so that condition\eqref{c2} holds, and $m_n$ in
\eqref{it3} is a parameter which measures the accuracy of the
finite-dimensional approximations $T^{(m_n)}$ and $K^*_{m_n}$ at the
$n-$th iteration. We propose a rule for choosing the parameters
$m_n$ so that $m_n$ depend on the parameters $a_n$. This rule yields
a non-decreasing sequence $m_n$. Since $m_n$ is a non-decreasing
sequence, we may start to compute $T_{a_n,m_n}^{-1}K_{m_n}^*f_\dl$
using a small size linear algebraic system \be\label{ag}
T_{a_n,m_n}g^\dl=K_{m_n}^*f_\dl,\ee  and increase the value of $m_n$
only if $G_{n,m_n}>C\dl^\varepsilon$, $C>2$, $\varepsilon\in(0,1)$,
where $G_{n,m_n}$ is defined below, in \eqref{Gnm}. Parameters $m_n$
may take large values for $n\leq n_\dl$, where $n_\dl$ is defined
below, in \eqref{srule}. The choice of the parameters $m_i$,
$i=1,2,\hdots$, in \eqref{it3}, which guarantees convergence of the
iterative process \eqref{it3}, is given in Section 2. We prove in
Section 3 that the discrepancy-type principle, proposed in
\cite{SWIAGR09}, with $T^{(m)}$ and $K_m^*$ in place of $T$ and
$K^*$ respectively, guarantees the convergence of the approximate
solution $u_{n,m_n}^\dl$ to the minimal norm solution of equation
\eqref{11}. Throughout this paper we assume that \be y\perp
\mathcal{N}(K),\ee and
\be Ky=f, \ee where $\mathcal{N}(K)$ is the nullspace of $K$.\\
Throughout this paper we denote by $K_m^*$ the operator approximating
$K^*$, and define
\be\label{Ta}
T_a:=T+aI,\quad T:=K^*K,
\ee
where $a=const>0$ and $I$ is the identity operator.

The main result of this paper is \thmref{MRT} in Section 3.

\section{Convergence of the iterative scheme}
In this section we derive sufficient conditions on the parameters
$m_i$, $i=1,2,\hdots,$ for the iterative process \eqref{it3} to
converge to the minimal-norm solution $y$. The estimates of the
following Lemma are known (see, e.g., \cite{RAMM499}), so their
proofs are omitted.
\begin{lem}\label{lemad1} One has:
\be\label{eq23}\|T_a^{-1}\|\leq\frac{1}{a}\ee and \be\label{eq24}
\|T_a^{-1}K^*\|\leq \frac{1}{2\sqrt{a}},\ee for any positive
constant $a$.
\end{lem}

 While $T_a$ is boundedly invertible for every $a>0$, $T_{a,m}$ may be
not invertible. The following
lemma provides sufficient conditions for $T_{a,m}$ to be boundedly invertible.
\begin{lem}\label{lemad2} Suppose that
\be\label{c3}\|T-T^{(m)}\|<\epsilon a, \quad a=const>0,\ee where
$\epsilon\in(0,1/2]$. Then the following
estimates hold
\be\label{ei}\|T_{a,m}^{-1}\|\leq
\frac{2}{a},\ee
\be\label{eii}
\|T_{a,m}^{-1}K^*\|\leq \frac{1}{\sqrt{a}} \ee and
\be\label{eiii}
\|T_{a,m}^{-1}K^*K\|\leq 2. \ee
\end{lem}
\begin{proof} Write \be\label{dsT}
T_{a,m}=T_a\left[I+T_a^{-1}(T^{(m)}-T) \right]. \ee It follows from
\eqref{c3} and \eqref{eq23} that \be\label{ds3}
\|T_a^{-1}(T^{(m)}-T)\|\leq \|T_a^{-1}\|\|(T^{(m)}-T)\|\leq
\epsilon<1. \ee
Therefore the operator $I+T_a^{-1}(T^{(m)}-T)$ is boundedly
invertible. Since $T_a$ is invertible, it follows from \eqref{dsT} and \eqref{ds3} that $T_{a,m}$ is invertible and
\be\label{ds4}
T_{a,m}^{-1}=\left[I+T_a^{-1}(T^{(m)}-T)\right]^{-1}T_a^{-1}. \ee
Let us estimate the norm $\|T_{a,m}^{-1}\|$. We have $0<\epsilon
\leq 1/2$, so  \be\label{esds3}
\left\|\left[I+T_a^{-1}(T^{(m)}-T)\right]^{-1}\right\|\leq
\frac{1}{1-\|T_a^{-1}(T^{(m)}-T)\|}\leq \frac{1}{1-\epsilon}\leq 2.
\ee This, together with \eqref{eq23} and \eqref{ds4}, yields \be
\|T_{a,m}^{-1}\|\leq \frac{2}{a}. \ee Thus, estimate \eqref{ei} is
proved. To prove estimate \eqref{eii}, write
$$T_{a,m}^{-1}K^*=\left[I+T_a^{-1}(T^{(m)}-T)\right]^{-1}T_a^{-1}K^*.$$
Using estimates \eqref{esds3} and \eqref{eq24}, one gets
$$\|T_{a,m}^{-1}K^*\|\leq \frac{1}{\sqrt{a}}$$ which
proves estimate \eqref{eii}. Let us derive estimate \eqref{eiii}. One has:
$$T_{a,m}^{-1}K^*K=\left[I+T_a^{-1}(T^{(m)}-T)\right]^{-1}T_a^{-1}K^*K.$$
Using the estimates $\|T_a^{-1}T\|\leq 1$ and \eqref{esds3}, one
obtains $$\|T_{a,m}^{-1}T\|\leq \frac{1}{1-\epsilon}\leq 2.$$ \lemref{lemad2} is proved.
\end{proof}

\begin{lem}\label{lemq} Let $g(x)$ be a continuous function on
$(0,\infty)$, $c>0$ and $q\in(0,1)$ be constants. If \be\label{g}
\lim_{x\to 0^+}g(x)=g(0):=g_0,\ee then \be\label{rel2} \lim_{n\to
\infty}\sum_{j=0}^{n-1}\left(q^{n-j-1}-q^{n-j}\right)g(cq^{j+1})=
g_0. \ee
\end{lem}
\begin{proof}
Let \be\label{wi}w_j^{(n)}:=q^{n-j}-q^{n+1-j},\quad w_j^{(n)}>0,\ee
and \be\label{Fn}F_l(n):=\sum_{j=1}^{l-1}w_j^{(n)}g(cq^{j}).\ee Then
\bee |F_{n+1}(n)-g_0|\leq |F_{l}(n)|+\left|\sum_{j=l}^n
w_j^{(n)}g(cq^{j})-g_0\right|.\eee Take $\epsilon>0$ arbitrary
small. For sufficiently large $l(\epsilon)$ one can choose
$n(\epsilon)$, such that \bee |F_{l(\epsilon)}(n)|\leq
\frac{\epsilon}{2},\ \forall n>n(\epsilon), \eee because
$\lim_{n\to\infty}q^n=0.$ Fix $l=l(\epsilon)$ such that
$|g(cq^j)-g_0|\leq \frac{\epsilon}{2}$ for $j>l(\epsilon)$. This is
possible because of \eqref{g}. One has \bee |F_{l(\epsilon)}(n)|\leq
\frac{\epsilon}{2},\ n>n(\epsilon) \eee and \bee\begin{split}
\left|\sum_{j=l(\epsilon)}^n w_j^{(n)}g(cq^{j})-g_0\right|&\leq
\sum_{j=l(\epsilon)}^{n}
w_j^{(n)}|g(cq^{j})-g_0|+|\sum_{j=l(\epsilon)}^{n}
w_j^{(n)}-1||g_0|\\
&\leq
\frac{\epsilon}{2}\sum_{j=l(\epsilon)}^nw_j^{(n)}+q^{n-l(\epsilon)}|g_0|\\
&\leq \frac{\epsilon}{2}+|g_0|q^{n-l(\epsilon)}\leq
\epsilon,\end{split}\eee if $n$ is sufficiently large. Here we have
used the relation\bee \sum_{j=l}^{n}w_j^{(n)}=1-q^{n+1-l}. \eee
Since $\epsilon>0$ is arbitrarily small, relation \eqref{rel2}
follows.\\\lemref{lemq} is proved.
\end{proof}

\begin{lem}\label{lemex}
Let \be\label{it1}u_n=qu_{n-1}+(1-q)T_{a_n}^{-1}K^*f,\quad u_0=0,\quad a_n:=\alpha_0q^n,\quad q\in(0,1).\ee Then \be\label{UB}
\|u_n-y\|\leq q^n\|y\|+\sum_{j=0}^{n-1}\left(
q^{n-j-1}-q^{n-j}\right)a_{j+1}\|T_{a_{j+1}}^{-1}y\|,\quad \forall
n\geq 1,\ee and \be\label{conu} \|u_n-y\|\to 0 \text{ as } n\to
\infty. \ee
\end{lem}
\begin{proof}
By induction, we obtain \be\label{dun}
u_n=\sum_{j=0}^{n-1}w_{j}^{(n)}T_{a_{j+1}}^{-1}K^*f, \ee where
$w_j^{(n)}=q^{n-j-1}-q^{n-j}$. This, together with the identities
$Ky=f$, \be T_a^{-1}K^*K=T_a^{-1}(K^*K+aI-aI)=I-aT_a^{-1} \ee and
\be\label{swi} \sum_{j=0}^nw_{j}^{(n)}=1-q^n, \ee yield
\bee\begin{split}
u_{n}&=\sum_{j=0}^{n-1}w_{j}^{(n)}T_{a_{j+1}}^{-1}(T_{a_{j+1}}-a_{j+1}I)y\\
&=\sum_{j=0}^{n-1}w_{j}^{(n)}y-\sum_{j=0}^{n-1}w_{j}^{(n)}a_{j+1}T_{a_{j+1}}^{-1}y\\
&=y-q^ny-\sum_{j=0}^{n-1}w_{j}^{(n)}a_{j+1}T_{a_{j+1}}^{-1}y.
\end{split}\eee Thus, estimate \eqref{UB} follows.
To prove \eqref{conu}, we apply \lemref{lemq} with
$g(a):=a\|T_a^{-1}y\|.$ Since $y\perp \mathcal{N}(K)$, it follows
from the spectral theorem that $$\lim_{a\to 0}g^2(a)=\lim_{a\to 0}
\int_0^\infty\frac{a^2}{(a+s)^2}d\langle
E_sy,y\rangle=\|P_{\mathcal{N}(K)}y\|^2=0,$$ where $E_s$ is the
resolution of the identity corresponding to $K^*K$, and $P$ is the
orthogonal projector onto $\mathcal{N}(K)$. Thus, by \lemref{lemq},
\eqref{conu} follows.\\
\lemref{lemex} is proved.
\end{proof}

\begin{lem}\label{lem21}
Let $u_n$ and $a_n=\alpha_0q^n$, $\alpha_0>0$, $q\in(0,1)$ be
defined in \eqref{it1}, $T_{a,m}$ be defined in \eqref{Tm}, $m_i$ be
chosen so that \be \|T-T^{(m_i)}\|\leq \frac{a_i}{2},\quad 1\leq
i\leq n, \ee and \be\label{it2}
u_{n,m_n}=qu_{n-1,m_{n-1}}+(1-q)T_{a_n,m_n}^{-1}K_{m_n}^*f,\quad
u_{0,m_0}=0. \ee Then \be\begin{split}\label{mes1}
\|u_{n,m_n}-u_n\|&\leq q^n\|y\|+
\|y-u_n\|+2\sum_{j=0}^{n-1}w_{j+1}^{(n)}\frac{\|(K_{m_{j+1}}^*K-
T^{(m_{j+1})})y\|}{a_{j+1}}\\
&+2\sum_{j=0}^{n-1}w_{j+1}^{(n)}a_{j+1}\|T_{a_{j+1}}^{-1}y\|,
\end{split}\ee where $w_j^{(n)}$ are defined in \eqref{wi}.\end{lem}
\begin{proof}
One has $ w_i^{(n)}>0$,  $0<q<1$, and
$$\sum_{j=0}^{n-1}w_{j+1}^{(n)}=1-q^n\to 1,\quad  as
\quad  n\to \infty.$$
Therefore one may use $w_{j+1}^{(n)}$ for large $n$ as quadrature
weights. To prove inequality \eqref{mes1}, the following lemma is
needed:
\begin{lem}\label{slem21}
Let $u_{n,m_n}$ be defined in \eqref{it2}. Then \be\label{sunm}
u_{n,m_n}=\sum_{j=0}^{n-1}w_{j+1}^{(n)}T_{a_{j+1},m_{j+1}}^{-1}K^*_{m_{j+1}}f,\quad
n>0, \ee where $w_j$ are defined in \eqref{wi}.
\end{lem}
\begin{proof}
Let us prove equation \eqref{sunm} by induction. For $n=1$ we get
\bee\begin{split}u_{1,m_1}&=qu_0+(1-q)T_{a_1,m_1}^{-1}K_{m_1}^*f=(1-q)T_{a_1,m_1}^{-1}K_{m_1}^*f\\
&=w_{1}^{(1)}T_{a_1,m_1}^{-1}K_{m_1}^*f,\end{split}\eee so equation
\eqref{sunm} holds. Suppose equation \eqref{sunm} holds for $1\leq
n\leq k$. Then \be\begin{split}
u_{k+1,m_{k+1}}&=qu_{k,m_k}+(1-q)T_{a_{k+1},m_{k+1}}^{-1}K_{m_{k+1}}^*f\\
&=q\sum_{j=0}^{k-1}w_{j+1}^{(k)}T_{a_{j+1},m_{j+1}}^{-1}K_{m_{j+1}}^*f+(1-q)T_{a_{k+1},m_{k+1}}^{-1}K_{m_{k+1}}^*f\\
&=\sum_{j=0}^{k-1}w_{j+1}^{(k+1)}T_{a_{j+1},m_{j+1}}^{-1}K_{m_{j+1}}^*f+w_{k+1}^{(k+1)}T_{a_{k+1},m_{k+1}}^{-1}K_{m_{k+1}}^*f\\
&=\sum_{j=0}^{k}w_{j+1}^{(k+1)}T_{a_{j+1},m_{j+1}}^{-1}K_{m_{j+1}}^*f.\end{split}
\ee Here we have used the identities $qw_j^{(n)}=w_j^{(n+1)}$ and
$1-q=w_j^{(j)}.$ Equation \eqref{sunm} is proved.
\end{proof}

By \lemref{slem21},
one gets: \bee\begin{split}
u_{n,m_n}-u_n&=\sum_{j=0}^{n-1}w_{j+1}^{(n)}T_{a_{j+1},m_{j+1}}^{-1}K_{m_{j+1}}^*Ky-u_n\\
&=\sum_{j=0}^{n-1}w_{j+1}^{(n)}T_{a_{j+1},m_{j+1}}^{-1}(K_{m_{j+1}}^*K-T^{(m_{j+1})}+T^{(m_{j+1})})y-u_n\\
&:=I_1+I_2,
\end{split}\eee
where
$$I_1:=\sum_{j=0}^{n-1}w_{j+1}^{(n)}T_{a_{j+1},m_{j+1}}^{-1}(K_{m_{j+1}}^*K-T^{(m_{j+1})}+T^{(m_{j+1})})y,$$
and
$$I_2:=-u_n.$$ We get \bee\begin{split}
I_1&=\sum_{j=0}^{n-1}w_{j+1}^{(n)}\left[T_{a_{j+1},m_{j+1}}^{-1}(K_{m_{j+1}}^*K-T^{(m_{j+1})})y+T_{a_{j+1},m_{j+1}}^{-1}T^{(m_{j+1})}y \right]\\
&=\sum_{j=0}^{n-1}w_{j+1}^{(n)}\left[T_{a_{j+1},m_{j+1}}^{-1}(K_{m_{j+1}}^*K-T^{(m_{j+1})})y+y-a_{j+1}T_{a_{j+1},m_{j+1}}^{-1}y \right]\\
&=\sum_{j=0}^{n-1}w_{j+1}^{(n)}\left[T_{a_{j+1},m_{j+1}}^{-1}(K_{m_{j+1}}^*K-T^{(m_{j+1})})y-a_{j+1}T_{a_{j+1},m_{j+1}}^{-1}y \right]\\
&+y-q^ny\\
&=\sum_{j=0}^{n-1}w_{j+1}^{(n)}T_{a_{j+1},m_{j+1}}^{-1}(K_{m_{j+1}}^*K-T^{(m_{j+1})})y\\
&-\sum_{j=0}^{n-1}w_{j+1}^{(n)}a_{j+1}(T_{a_{j+1},m_{j+1}}^{-1}-T_{a_{j+1}}^{-1}+T_{a_{j+1}}^{-1})y+y-q^ny\\
&=y-q^ny+\sum_{j=0}^{n-1}w_{j+1}^{(n)}T_{a_{j+1},m_{j+1}}^{-1}(K_{m_{j+1}}^*K-T^{(m_{j+1})})y\\
&-\sum_{j=0}^{n-1}w_{j+1}^{(n)}\left[a_{j+1}T_{a_{j+1},m_{j+1}}^{-1}(T-T^{(m_{j+1})})T_{a_{j+1}}^{-1}y
+a_{j+1}T_{a_{j+1}}^{-1}y\right].
\end{split}\eee Therefore,
\be\begin{split}
I_1+I_2&=y-u_n-q^ny+\sum_{j=0}^{n-1}w_{j+1}^{(n)}T_{a_{j+1},m_{j+1}}^{-1}(K_{m_{j+1}}^*K-T^{(m_{j+1})})y\\
&-\sum_{j=0}^{n-1}w_{j+1}^{(n)}\left[a_{j+1}T_{a_{j+1}}^{-1}+a_{j+1}T_{a_{j+1},m_{j+1}}^{-1}(T-T^{(m_{j+1})})T_{a_{j+1}}^{-1}\right]y.
\end{split}\ee Applying the estimates $\|T^{(m_i)}-T\|\leq \frac{a_i}{2}$ and
$\|T_{a_{i},m_{i}}^{-1}\|\leq \frac{2}{a_i}$ in \eqref{es1a}, one
gets
\be\begin{split}\label{es1a} \|u_{n,m}-u_n\|&\leq q^n\|y\|+
\|y-u_n\|+\sum_{j=0}^{n-1}w_{j+1}^{(n)}a_{j+1}\|T_{a_{j+1}}^{-1}y\|\\
&+\sum_{j=0}^{n-1}w_{j+1}^{(n)}\|T_{a_{j+1},m_{j+1}}^{-1}(K_{m_{j+1}}^*K-T^{(m_{j+1})})y\|\\
&+\sum_{j=0}^{n-1}w_{j+1}^{(n)}\|a_{j+1}T_{a_{j+1},m_{j+1}}^{-1}(T-T^{(m_{j+1})})T_{a_{j+1}}^{-1}y\|\\
&\leq q^n\|y\|+
\|y-u_n\|+\sum_{j=0}^{n-1}w_{j+1}^{(n)}a_{j+1}\|T_{a_{j+1}}^{-1}y\|\\
&+\sum_{j=0}^{n-1}w_{j+1}^{(n)}\|T_{a_{j+1},m_{j+1}}^{-1}\|\|(K_{m_{j+1}}^*K-T^{(m_{j+1})})y\|\\
&+\sum_{j=0}^{n-1}w_{j+1}^{(n)}a_{j+1}\|T_{a_{j+1},m_{j+1}}^{-1}\|\|T-T^{(m_{j+1})}\|\|T_{a_{j+1}}^{-1}y\|\\
&\leq q^n\|y\|+
\|y-u_n\|+\sum_{j=0}^{n-1}w_{j+1}^{(n)}a_{j+1}\|T_{a_{j+1}}^{-1}y\|\\
&+\sum_{j=0}^{n-1}w_{j+1}^{(n)}\frac{2}{a_{j+1}}\|(K_{m_{j+1}}^*K-T^{(m_{j+1})})y\|\\
&+\sum_{j=0}^{n-1}w_{j+1}^{(n)}a_{j+1}\|T_{a_{j+1}}^{-1}y\|.
\end{split}\ee\lemref{lem21} is proved.
\end{proof}
\begin{lem}\label{lem22}
Under the assumptions of \lemref{lem21} if \be \|K_{m_n}^*-K^*\|\leq
\frac{\sqrt{a_n}}{2} \ee then \be\label{mes2}
\|u_{n,m_n}-u_{n,m_n}^\dl\|\leq
\frac{\sqrt{q}}{1-q^{3/2}}\frac{2\dl}{\sqrt{q}\sqrt{a_n}}. \ee
\end{lem}
\begin{proof}
We have \be\begin{split}
&u_{n,m_n}-u_{n,m_n}^\dl=q(u_{n-1,m_{n-1}}-u_{n-1,m_{n-1}}^\dl)
+(1-q)T_{a_n,m_n}^{-1}K_{m_n}^*(f-f_\dl)\\
&=q(u_{n-1,m_{n-1}}-u_{n-1,m_{n-1}}^\dl)
+(1-q)T_{a_n,m_n}^{-1}(K_{m_n}^*-K^*)(f-f_\dl)\\
&+(1-q)T_{a_n,m_n}^{-1}K^*(f-f_\dl).\end{split}\ee Since
$\|f-f_\dl\|\leq \dl$, $\|T_{a_n,m_n}^{-1}K^*\|\leq
\frac{1}{\sqrt{a_n}}$ and $\|K_{m_n}^*-K^*\|\leq
\frac{\sqrt{a_n}}{2}$, it follows that \be
\|u_{n,m_n}-u_{n,m_n}^\dl\|\leq
q\|u_{n-1,m_{n-1}}-u_{n-1,m_{n-1}}^\dl\|+2\frac{\dl}{\sqrt{a_n}}.
\ee Let us prove estimate \eqref{mes2} by induction. Define
$H_n:=\|u_{n,m_n}-u_{n,m_n}^\dl\|$ and
$h_n:=2\frac{\dl}{\sqrt{q}\sqrt{a_n}}.$ For $n=0$ we get $H_0=0<
\frac{\sqrt{q}}{1-q^{3/2}}h_0.$ Thus \eqref{mes2} holds. Suppose
estimate \eqref{mes2} holds for $0\leq n\leq k$. Then
\be\begin{split} H_{k+1}&\leq qH_k+h_{k}\leq q
\frac{\sqrt{q}}{1-q^{3/2}}h_k+h_{k}=\left(q
\frac{\sqrt{q}}{1-q^{3/2}}+1\right)h_k\\
&=\frac{1}{1-q^{3/2}}\frac{h_k}{h_{k+1}}h_{k+1}\leq
\frac{\sqrt{q}}{1-q^{3/2}}h_{k+1}.
\end{split}\ee Here we have used the relation
\be
\frac{h_k}{h_{k+1}}=\frac{2\frac{\dl}{\sqrt{q}\sqrt{a_k}}}{2\frac{\dl}{\sqrt{q}\sqrt{a_{k+1}}}}=\frac{\sqrt{a_{k+1}}}{\sqrt{a_k}}=\frac{\sqrt{qa_k}}{\sqrt{a_k}}=\sqrt{q}.
\ee \lemref{lem22} is proved.
\end{proof}

The following theorem gives the convergence of the iterative scheme
\eqref{it3}.
\begin{thm}\label{thm21}
Let $u_{n,m_{n}}^\dl$ be defined in \eqref{it3}, $m_i$ be chosen so
that \be\label{mcon1} \|T-T^{(m_i)}\|\leq a_i/2,\ee \be\label{mcon2}
\|T^{(m_i)}-K_{m_i}^*K\|\leq a_i^2,\ee
\be\label{mcon3}\|K_{m_i}^*-K^*\|\leq \sqrt{a_i}/2, \ee and $n_\dl$
satisfies the following relations: \be\label{asdel} \lim_{\dl\to
0}n_\dl=\infty,\quad \lim_{\dl\to 0}\frac{\dl}{\sqrt{a_{n_\dl}}}=0.
\ee Then \be\label{mcuy} \lim_{\delta \to
0}\|u_{n_\dl,m_{n_\dl}}^\dl -y\|= 0. \ee
\end{thm}
\begin{proof}
We have \be \|y-u_{n,m_n}^\dl\|\leq
\|y-u_n\|+\|u_n-u_{n,m_n}\|+\|u_{n,m_n}-u_{n,m_n}^\dl\|. \ee
 From \eqref{mes1} and estimate \eqref{mcon2} we get \be \|u_{n,m_n}-u_n\|\leq
q^n\|y\|+\|y-u_n\|+2\sum_{j=0}^{n-1}w_{j+1}^{(n)}a_{j+1}\|y\|+2\sum_{j=0}^{n-1}w_{j+1}^{(n)}a_{j+1}\|T_{a_{j+1}}^{-1}y\|.
 \ee
This, together with \lemref{lem22}, implies \be\label{et27}
\|y-u_{n,m_n}^\dl\|\leq
2\left(J(n)+\frac{\dl}{(1-q^{3/2})\sqrt{a_n}}\right),
  \ee where \be\label{Jn}J(n):=\frac{q^n}{2}\|y\|+
\|y-u_n\|+\sum_{j=0}^{n-1}w_{j+1}^{(n)}a_{j+1}\|y\|+\sum_{j=0}^{n-1}w_{j+1}^{(n)}a_{j+1}\|T_{a_{j+1}}^{-1}y\|,\ee
and $w_j^{(n)}$ are defined in \eqref{wi}.
 Since $y\perp
\mathcal{N}(A)$, it follows that $$\lim_{a\to 0}
a^2\|T_a^{-1}y\|^2=\int_0^\infty \frac{a^2}{(a+s)^2}d\la
E_sy,y\ra=\|P_{\mathcal{N}(K)}y\|^2=0,$$ where $E_s$ is the
resolution of the identity of the selfadjoint operator $T$, and
$P_{\mathcal{N}(K)}$ is the orthogonal projector onto the nullspace
 ${\mathcal{N}(K)}$. Applying \lemref{lemq} with $g(a):=a\|T_a^{-1}y\|$, one gets
\be\label{gaT} \lim_{n\to
\infty}\sum_{j=0}^{n-1}w_{j+1}^{(n)}a_{j+1}\|T_{a_{j+1}}^{-1}y\|=0.
\ee Similarly, letting $g(a):=a\|y\|$ in \lemref{lemq}, we get
\be\label{gy} \lim_{n\to
\infty}2\sum_{j=0}^{n-1}w_{j+1}^{(n)}a_{j+1}\|y\|=0. \ee Relations
\eqref{gaT} and \eqref{gy}, together with \lemref{lemex}, imply \be
\lim_{n\to \infty} J(n)=0. \ee If we stop the iteration at $n=n_\dl$
such that assumptions \eqref{asdel} hold then $\lim_{\dl\to 0}
J(n_\dl)=0$ and $\lim_{\dl\to 0}\frac{\dl}{\sqrt{a_{n_\dl}}}=0.$
Therefore, relation \eqref{mcuy} is proved.
This proves \thmref{thm21}.
\end{proof}
\section{ A discrepancy-type principle for DSM}
In this section we propose an adaptive stopping rule for the
iterative scheme \eqref{it3}. Throughout this section the parameters
$m_i,\ i=1,2,\hdots, $ are chosen so that conditions
\eqref{mcon1}-\eqref{mcon3} hold, \be\label{cq1} \|Q-Q^{(m_i)}\|\leq
\epsilon a_i,\quad \epsilon\in(0,1/2],\quad a_i=\alpha_0q^i,\quad
\alpha_0=const>0, \ee where \be Q:=KK^*, \ee and $Q^{(m)}$ is a
finite-dimensional approximation of $Q.$  One may satisfy condition
\eqref{cq1} by approximating the kernel $q(x,s)$ of $Q$,
\be\label{oq} q(x,s)=\int_a^bk(x,z)\overline{k(s,z)}dz, \ee with
\be\label{qm}
q_m(x,s)=\sum_{i=1}^m\gamma_ik(x,z_i)\overline{k(s,z_i)}, \ee where
$\gamma_i,\ i=1,2,\hdots, m,$ are some quadrature weights and $z_i$
are the collocation points.
\begin{lem}\label{lemq1}
\be\label{cq2}\|Q_a^{-1}\|\leq\frac{1}{a}\ee and \be\label{cq3}
\|Q_a^{-1}K\|\leq \frac{1}{2\sqrt{a}},\ee for any positive constant
$a$.
\end{lem}
\begin{proof} Since $Q=Q^*\geq 0$, one uses the spectral theorem and
gets:
$$\|Q_a^{-1}\|=\sup_{s>0}\frac{1}{s+a}\leq
\frac{1}{a}.$$ Inequality \eqref{cq3} follows from the identity
\be\label{cq4}Q_a^{-1}K=KT_a^{-1},\quad T:=K^*K,\quad T_a:=T+aI, \ee
and the estimate \be \|KT_a^{-1}\|=\|UT^{1/2}T_a^{-1}\|\leq
\|T^{1/2}T_a^{-1}\|=\sup_{s\geq 0}\frac{s^{1/2}}{a+s}\leq
\frac{1}{2\sqrt{a}}, \ee where the polar decomposition was used:
$K=UT^{1/2}$, $U$ is a partial isometry, $\|U\|=1$.  \lemref{lemq1}
is proved.
\end{proof}

\begin{lem}\label{lemq2}
Suppose $m$ is chosen so that \be\label{mcq} \|Q-Q^{(m)}\|\leq
\epsilon a,\quad \epsilon\in(0,1/2],\quad a>0. \ee Then the
following estimates hold: \be\label{esa} \|Q_{a,m}^{-1}\|\leq
\frac{2}{a},\ee \be\label{esb} \|Q_{a,m}^{-1}K\|\leq
\frac{1}{\sqrt{a}}.\ee
\end{lem}Proof of \lemref{lemq2} is similar to the proof of \lemref{lemad2}
and is omitted.\\

We propose the following \textit{stopping rule}:\\Choose $n_\dl$ so
that the following inequalities hold \be\label{srule}
G_{n_\dl,m_{n_\dl}}\leq C\dl^\varepsilon<G_{n,m_n},\quad 1\leq
n<n_\dl,\ C>2,\quad \varepsilon\in(0,1), \ee \textit{where}
\be\begin{split}\label{Gnm}
G_{n,m_n}&=qG_{n-1,m_{n-1}}+(1-q)a_n\|Q_{a_n,m_n}^{-1}f_\dl\|,\\
G_{0,m_0}&=0,\quad G_{1,m_1}\geq C\dl^\varepsilon,\quad
a_n=qa_{n-1},\quad a_0=\alpha_0=const>0,\end{split}\ee \textit{and}
\be\label{Qam} Q_{a,m}:=Q^{(m)}+aI.\ee The discrepancy-type
principle \eqref{srule} is derived from the following discrepancy
principle for DSM proposed in \cite{RAMM525a, RAMM525b}: \be
\int_0^{t_\dl}e^{-(t_\dl-s)}a(s)\|Q_{a(s)}^{-1}f_\dl\|ds=C\dl,\quad
C>1, \ee where $t_\dl$ is the stopping time, and we assume that
$$a(t)>0,\quad a(t)\searrow 0.$$ The derivation of the stopping rule
\eqref{srule} with $Q^{(m)}=Q$ is given in \cite{SWIAGR09}. Let us
prove that there exists an integer $n_\dl$ such that inequalities
\eqref{srule} hold. To prove the existence of such an integer, we
derive some properties of the sequence $G_{n,m_n}$ defined in
\eqref{Gnm}. Using \lemref{lemq2}, the relation $Ky=f$, and the
assumption $\|f_\dl-f\|\leq \dl$, we get \be\begin{split}
a_n\|Q_{a_n,m_n}^{-1}f_\dl\|&\leq
a_n\|Q_{a_n,m_n}^{-1}(f_\dl-f)\|+a_n\|Q_{a_n,m_n}^{-1}f\|\\
&\leq 2\dl+2\sqrt{a_n}\|y\|,
\end{split}\ee where estimates \eqref{esa} and \eqref{esb} were
used. This, together with \eqref{Gnm}, yield \be G_{n,m_n}\leq
qG_{n-1,m_{n-1}}+(1-q)2\dl+(1-q)2\sqrt{a_n}\|y\|, \ee so \be
G_{n,m_n}-2\dl\leq
q(G_{n-1,m_{n-1}}-2\dl)+(1-q)2\sqrt{q}\sqrt{a_{n-1}}\|y\|, \ee where
the relation $a_n=qa_{n-1}$, $a_0=\alpha_0=const>0$, was used.
Define \be\label{dHn}\Psi_n:=G_{n,m_n}-2\dl,\ee where $G_{n,m}$ is
defined in \eqref{Gnm}, and let
\be\label{dwn}\psi_n:=(1-q)2\sqrt{a_{n}}\|y\|.\ee Then \be
\Psi_n\leq q\Psi_{n-1}+\sqrt{q}\psi_{n-1}.\ee
\begin{lem}\label{lem3}
If  \eqref{dHn} and \eqref{dwn} hold,
then
\be\label{eHn} \Psi_n\leq
\frac{1}{1-\sqrt{q}}\psi_n,\quad n\geq 0.\ee
\end{lem}
\begin{proof}
Let us prove this lemma by induction. For $n=0$ we get
$$\Psi_0=-2\dl\leq \frac{1}{1-\sqrt{q}}\psi_0. $$ Suppose estimate
\eqref{eHn} is true for $0\leq n\leq k. $ Then \be\begin{split}
\Psi_{k+1}&\leq q\Psi_k+\sqrt{q}\psi_k\leq
\frac{q}{1-\sqrt{q}}\psi_k+\sqrt{q}\psi_k=\frac{\sqrt{q}}{1-\sqrt{q}}\psi_k\\
&=\frac{\sqrt{q}}{1-\sqrt{q}}\frac{\psi_k}{\psi_{k+1}}\psi_{k+1}\leq
\frac{\sqrt{q}}{1-\sqrt{q}}\frac{1}{\sqrt{q}}\psi_{k+1}=\frac{1}
{1-\sqrt{q}}\psi_{k+1}.
\end{split}\ee
Here we have used the relation
\be
\frac{\psi_k}{\psi_{k+1}}=\frac{(1-q)2\sqrt{a_k}\|y\|}
{(1-q)2\sqrt{a_{k+1}}\|y\|}=\frac{\sqrt{a_k}}{\sqrt{a_{k+1}}}=
\frac{\sqrt{a_k}}{\sqrt{qa_k}}=\frac{1}{\sqrt{q}}.
\ee

Thus, \lemref{lem3} is proved.
\end{proof}
By definitions \eqref{dHn}, \eqref{dwn}, and \lemref{lem3}, we get
the estimate \be\label{MeGm} G_{n,m_n}\leq
2\dl+\frac{1}{1-\sqrt{q}}(1-q)2\sqrt{a_{n}}\|y\|,\quad n\geq 0,\ee
so \be \limsup_{n\to \infty} G_{n,m_n}\leq 2\dl \ee because
$\lim_{n\to \infty} a_n=0.$\\ Since $G_{1,m_1}\geq
C\dl^\varepsilon$, $C>2$, $\varepsilon\in(0,1)$ and $\limsup_{n\to
\infty}G_{n,m_n}\leq 2\dl$, it follows that there exists an integer
$n_\dl$ such that inequalities \eqref{srule} hold. The uniqueness of
the integer $n_\dl$ follows
from its definition.\\
\begin{lem}\label{lem4}
If $n_\dl$ is chosen by the rule \eqref{srule}, then \be
\frac{\dl}{\sqrt{a_{n_\dl}}}\to 0 \text{ as }\dl\to 0.\ee
\end{lem}
\begin{proof}
From the stopping rule \eqref{srule} and estimate \eqref{MeGm} we
get \be C\dl^\varepsilon<G_{n_\dl-1,m_{n_\dl-1}}\leq 2\dl +
\frac{1}{1-\sqrt{q}}(1-q)2\sqrt{a_{n_\dl-1}}\|y\|. \ee This implies
\be \frac{1}{ \sqrt{a_{n_\dl-1}}}\leq
\frac{1}{(1-\sqrt{q})(C-2)\dl^\varepsilon}(1-q)2\|y\|, \ee so \be
\frac{\dl}{\sqrt{a_{n_\dl}}}\leq
\frac{\dl^{1-\varepsilon}}{\sqrt{q}(1-\sqrt{q})(C-2)}(1-q)2\|y\|\to
0 \text{ as } \dl\to 0. \ee \lemref{lem4} is proved.
\end{proof}
\begin{lem}\label{lem5}
If $n_\dl$ is chosen by the rule \eqref{srule}, then
\be\label{rel1}\lim_{\dl\to 0} n_\dl= \infty.\ee\end{lem}
\begin{proof}
 From the stopping rule \eqref{srule} we get
\be\begin{split}
&qC\dl^\varepsilon+(1-q)a_{n_\dl}\|Q_{a_{n_\dl},m_{n_\dl}}^{-1}f_\dl\|<
qG_{n_\dl-1,m_{n_\dl-1}}+(1-q)a_{n_\dl}\|Q_{a_{n_\dl},m_{n_\dl}}^{-1}f_\dl\|\\
&=G_{n_\dl,m_{n_\dl}}<C\dl^\varepsilon.
\end{split}\ee This implies \be\label{mes3} 0\leq
a_{n_\dl}\|Q_{a_{n_\dl},m_{n_\dl}}^{-1}f_\dl\|<C\dl^\varepsilon\to 0
\text{ as } \dl\to 0. \ee Note that \be\begin{split} 0&\leq
a_{n_\dl}\|Q_{a_{n_\dl}}^{-1}f_\dl\|\leq
a_{n_\dl}\|(Q_{a_{n_\dl}}^{-1}-Q_{a_{n_\dl},m_{n_\dl}}^{-1})f_\dl\|
+a_{n_\dl}\|Q_{a_{n_\dl},m_{n_\dl}}^{-1}f_\dl\|\\
&=a_{n_\dl}\|Q_{a_{n_\dl}}^{-1}(Q_{a_{n_\dl},m_{n_\dl}}-Q_{a_{n_\dl}})Q_{a_{n_\dl},m_{n_\dl}}^{-1}f_\dl\|
+a_{n_\dl}\|Q_{a_{n_\dl},m_{n_\dl}}^{-1}f_\dl\|\\
&=a_{n_\dl}\|Q_{a_{n_\dl}}^{-1}(Q^{(m_{n_\dl})}-Q)Q_{a_{n_\dl},m_{n_\dl}}^{-1}f_\dl\|
+a_{n_\dl}\|Q_{a_{n_\dl},m_{n_\dl}}^{-1}f_\dl\|\\
&\leq
a_{n_\dl}\|Q_{n_\dl}^{-1}\|\|Q^{(m_{n_\dl})}-Q\|\|Q_{a_{n_\dl},m_{n_\dl}}^{-1}f_\dl\|
+a_{n_\dl}\|Q_{a_{n_\dl},m_{n_\dl}}^{-1}f_\dl\|\\
&\leq a_{n_\dl}\frac{2}{a_{n_\dl}}\epsilon
a_{n_\dl}\|Q_{a_{n_\dl},m_{n_\dl}}^{-1}f_\dl\|
+a_{n_\dl}\|Q_{a_{n_\dl},m_{n_\dl}}^{-1}f_\dl\|\\
&\leq 2a_{n_\dl}\|Q_{a_{n_\dl},m_{n_\dl}}^{-1}f_\dl\|,\\
\end{split}\ee where estimates \eqref{cq1}, \eqref{esa} and
$0<\epsilon<\frac{1}{2}$ were used. This, together with
\eqref{mes3}, yield \be\label{mes4} \lim_{\dl\to
0}a_{n_\dl}\|Q_{a_{n_\dl}}^{-1}f_\dl\|=0. \ee To prove relation
\eqref{rel1} the following lemma is needed:

\begin{lem}\label{lemand}Suppose condition $\|f-f_\dl\|\leq \dl$ and relation \eqref{mes4} hold. Then \be\label{dan} \lim_{\dl\to 0}a_{n_\dl}=
0.\ee
\end{lem}\begin{proof}
If $f\neq 0$ then there exists a $\lambda_0>0$ such that
\be\label{Flo} F_{\lambda_0}f\neq 0,\quad \langle
F_{\lambda_0}f,f\rangle:=\xi>0, \ee where $\xi$ is a constant which
does not depend on $\dl$, and $F_s$ is the resolution of the
identity corresponding to the operator $Q:=KK^*$. Let
$$h(\dl,\alpha):=\alpha^2\|Q_{\alpha}^{-1}f_\dl\|^2,\quad Q:=KK^*,\ Q_a:=aI+Q.$$ For a
fixed number $c_1>0$ we obtain \be\begin{split}
h(\dl,c_1)&=c_1^2\|Q_{c_1}f_\dl\|^2=\int_0^\infty
\frac{c_1^2}{(c_1+s)^2}d\langle F_sf_\dl,f_\dl\rangle\geq
\int_0^{\lambda_0} \frac{c_1^2}{(c_1+s)^2}d \langle
F_sf_\dl,f_\dl\rangle\\
&\geq \frac{c_1^2}{(c_1+\lambda_0)^2}\int_0^{\lambda_0} d\langle
F_sf_\dl,f_\dl\rangle=\frac{c_1^2\|F_{\lambda_0}f_\dl\|^2}{(c_1+\lambda_0)^2},\quad
\dl>0.
\end{split}\ee
Since $F_{\lambda_0}$ is a continuous operator, and
$\|f-f_\dl\|<\dl$, it follows from \eqref{Flo} that \be \lim_{\dl\to
0}\langle F_{\lambda_0}f_\dl,f_\dl\rangle=\langle
F_{\lambda_0}f,f\rangle>0 .\ee Therefore, for the fixed number
$c_1>0$ we get \be\label{hc1} h(\dl,c_1)\geq c_2>0\ee for all
sufficiently small $\dl>0$, where $c_2$ is a constant which does not
depend on $\dl$. For example one may take $c_2=\frac{\xi}{2}$
provided that \eqref{Flo} holds. It follows from relation
\eqref{mes4} that \be\label{limhq}\lim_{\dl\to 0} h(\dl,
a_{n_\dl})=0.\ee Suppose $\lim_{\dl\to 0}a_{n_\dl}\neq 0$. Then
there exists a subsequence $\dl_j\to 0$ such that \be
\alpha_0a_{n_{\dl_j}}\geq c_1>0, \ee where $c_1$ is a constant. By
\eqref{hc1} we get\be\label{hdj} h(\dl_j,a_{n_{\dl_j}})>c_2>0,\quad
\dl_j\to 0\text{ as }j\to \infty. \ee This contradicts relation
\eqref{limhq}. Thus, $\lim_{\dl\to0} a_{n_\dl}=0.$\\ \lemref{lemand}
is proved.
\end{proof}Applying \lemref{lemand} with $a_{n_\dl}=\alpha_0 q^{n_\dl}$, $q\in(0,1)$, $\alpha_0>0$, one gets relation \eqref{rel1}.\\
\lemref{lem5} is proved.
\end{proof}
We formulate the main result of this paper in the following theorem:
\begin{thm}\label{MRT}
Suppose $m_i$ are chosen so that conditions
\eqref{mcon1}-\eqref{mcon3} and \eqref{cq1} hold, and $n_\dl$ is
chosen by rule \eqref{srule}. Then \be\label{MR} \lim_{\dl\to
0}\|u_{n_\dl,m_{n_\dl}}^\dl-y\|=0. \ee
\end{thm}
\begin{proof}
From \eqref{et27} we get the estimate \be
\|y-u_{n_\dl,m_{n_\dl}}^\dl\|\leq
2\left(J(n_\dl)+\frac{\dl}{(1-q^{3/2})\sqrt{a_{n_\dl}}}\right),
  \ee where $J(n)$ is defined in \eqref{Jn}. It is proved in \thmref{thm21} that $\lim_{n\to \infty}J(n)=0.$
By \lemref{lem5}, one gets $n_\dl\to \infty$ as $\dl\to 0$, so
$\lim_{\dl\to 0}J(n_\dl)= 0$. From \lemref{lem4} we get
$\lim_{\dl\to 0}\frac{\dl}{\sqrt{a_{n_\dl}}}=0.$ Thus,
$$\lim_{\dl\to 0}\|y-u_{n_\dl,m_{n_\dl}}^\dl\|=0.$$ \thmref{MRT} is
proved.
\end{proof}
\section{Numerical experiments}
Consider the following Fredholm integral equation: \be\label{ILT}
Ku(s):=\int_{0}^1e^{-st}u(t)dt=f(s),\quad s\in[0,1].\ee
 The function
$u(t)=t$ is the solution to equation \eqref{ILT} corresponding to
$f(s)=\frac{1-(s+1)e^{-s}}{s^2}$. We perturb the exact data $f(s)$
by a random noise $\dl $, $\dl>0$, and get the noisy data
$f_\dl(s)=f(s)+\dl $. The compound Simpson's rule (see
\cite{PDPR84}) with the step size $\frac{1}{2^{m}}$ is used to
approximate the kernel $g(x,z),$ defined in \eqref{gxz}. This yields
$$T^{(m)}u:=\sum_{j=1}^{2^m+1}\beta_j^{(m)} k(s_j,x)\int_0^1
k(s_j,z)u(z)dz, $$ where $k(s,t):=e^{-st},$ $\beta_j^{(m)}$ are the
compound Simpson's quadrature weights:
$\beta_1^{(m)}=\beta_{2^m+1}^{(m)}=\frac{1/3}{2^m},$ and for
$j=2,3,\hdots, 2^m$ \be\beta_j^{(m)}=\left\{
                 \begin{array}{ll}
                   \frac{4/3}{2^m}, & \hbox{$j$ is even;} \\
                   \frac{2/3}{2^m}, & \hbox{otherwise,}
                 \end{array}
               \right.\ee

 and $s_j$ are the collocation
points: $s_j=\frac{j-1}{2^m}$, $j=1,2,\hdots,2^m+1$.\\ Let
$$\gamma_m:=\|(T-T^{(m)})u\|,$$ $$h(s,x,z):=k(s,x)k(s,z)$$ and
\be\label{cS}c_1:=\frac{1}{180}\max_{x,z\in[0,1]}\max_{s\in[0,1]}\left|\frac{\partial^4h(s,x,z)}{\partial
s^4}\right|=\frac{16}{180}.\ee
Then \be\begin{split}
\gamma_m^2&=\int_0^1 \left| \int_0^1\left(\int_0^1 h(s,x,z)ds-\sum_{j=1}^{2^m+1}\beta_j^{(m)}h(s_j,x,z)\right)u(z)dz \right|^2dx\\
&\leq \int_0^1 \left| \int_0^1\frac{c_1}{2^{4m}} u(z)dz \right|^2dx\leq \left(\frac{c_1}{2^{4m}}\right)^2\|u\|^2.\\
\end{split}\ee
The upper bound $c_1$ for the error
of the compound Simpson's quadrature can be found in \cite{PDPR84}.
Thus,
$$\|T-T^{(m)}\|\leq \frac{c_1}{2^{4m}}\to 0\text{ as }m\to \infty.$$
Similarly, we approximate the kernel $q(x,s)$ defined in \eqref{oq}
by the Simpson's rule with the step size $\frac{1}{2^m}$ and get \be
\|Q-Q^{(m)}\|\leq \frac{c_1}{2^{4m}}\to 0\text{ as }m\to \infty. \ee
 Let us partition the interval $[0,1]$ into $2^m180$, $m>0$,
equisized subintervals $D_j$, where $D_j=[d_{j-1},d_j),$
$j=1,2,\hdots, 2^m$. Then $|d_j-d_{j-1}|=\frac{1}{2^m180},$ $j=1,2,
\hdots, 2^m,$ and using the Taylor expansion of $e^{st}$ about
$s=d_{j-1}$, one gets \be\begin{split}\label{les}
&|e^{-st}-e^{-d_{j-1}t}[1-t(s-d_{j-1})]|\leq
\sum_{l=2}^\infty\frac{(s-d_j)^l}{l!}\leq (s-d_{j-1})^2\sum_{j=0}^\infty(s-d_{j-1})^j\\
&\leq \frac{1}{2^{2m}180^2}\sum_{j=0}^\infty\left(\frac{1}{2^{m}180}\right)^j= \frac{1}{2^{2m}180^2}\frac{2^m180}{2^m180-1}\\
&=\frac{1}{2^m180(2^{m}180-1)}\leq \frac{1}{2^{2m}180},\quad \forall s\in D_j,\ t\in[0,1].\\
\end{split}\ee This allows us to define

\be
K_m^*u(t)=\sum_{j=1}^{2^m}\int_{D_j}e^{-d_{j-1}t}[1-t(s-d_{j-1})]u(s)ds.\ee
This, together with condition \eqref{les}, yields \be\begin{split}
\|(K^*-K_m^*)u\|^2&=\int_0^1\left|\sum_{j=1}^{2^m}\int_{D_j}\left(e^{-st}-e^{-d_{j-1}t}[1-t(s-d_{j-1})]\right)u(t)dt\right|^2ds\\
&\leq\frac{1}{2^{2m}180^2}\int_0^1\left|\sum_{j=1}^{2^m}\int_{D_j}|u(t)|dt\right|^2ds\leq
\frac{1}{2^{4m}180^2}\|u\|^2.
\end{split}\ee Thus,
\be \|K^*-K_m^*\|\leq \frac{1}{2^{2m}180}\to 0\quad as\ m\to \infty.
\ee Moreover \be\begin{split} \|(T^{(m)}-K_m^*K)u\|&\leq
\|(T^{(m)}-T)u\|+\|(T-K_m^*K)u\|\\
&\leq \frac{c_1}{2^{4m}}\|u\|+\|K^*-K_m^*\|\|Ku\|\\
&\leq \frac{16}{2^{4m}180}\|u\|+\frac{1}{2^{2m}180}\|u\|\leq
\frac{17}{2^{2m}180}\|u\|.
\end{split}\ee Here we have used the constant $c_1=16/180$ and the estimate $|k(s,t)|\leq \max_{s,t\in[0,1]}|e^{-st}|=1.$ Thus,
\be \|T^{(m)}-K_m^*K\|\leq \frac{17}{2^{2m}180}. \ee To satisfy
condition \eqref{mcon1} the parameter $m_i$ may be chosen by solving
the equation \be \frac{c_1}{2^{4m_i}}=\frac{a_i}{2}. \ee To get
$m_i$ satisfying condition \eqref{mcon2}, one solves the equation
\be\label{numc2} \frac{17}{2^{2m_i}180}=\eta a_i^2, \ee where
$\eta=const\geq 10$. Here we have used the estimate
$\|T^{(m_i)}-K_{m_i}^*K\|\leq \eta a_i^2$ instead of estimate
\eqref{mcon2}. This estimate will not change our main results. The
reason of using the constant $\eta\geq 10$ than of $1$ in
\eqref{numc2} is to control the decaying rate of the parameter
$a_i^2$ so that the growth rate of the parameter $m_i$ in
\eqref{numc2} can be made as slow as we wish. To obtain the
parameter $m_i$ satisfying condition \eqref{mcon3}, one solves \be
\frac{c_1}{2^{2m_i}}=\frac{\sqrt{a_i}}{2}. \ee Hence to satisfy all
the conditions in \thmref{MRT}, one may choose $m_i$ such that
\be\label{rmi} m_i:=\max\left\{\left\lceil \frac{\ln
(2c_1/a_i)}{4\ln 2}\right\rceil,
\left\lceil\frac{\ln(\frac{17}{180(\eta a_i^2)})}{2\ln
2}\right\rceil,\left\lceil \frac{\ln (2c_1/\sqrt{a_i})}{2\ln
2}\right\rceil\right\},\ee where $\lceil x\rceil$ is the smallest
integer not less than $x$, $c_1$ is defined in \eqref{cS},
$a_i=\alpha_0q^i,\quad \alpha_0>0,\quad q\in(0,1).$ In all the
experiments the parameter $\eta$ in \eqref{rmi} is equal to $10$
which is sufficient for the given problem.
 To obtain the approximate solution to problem \eqref{ILT}, we
consider a finite-dimensional approximate solution
\be\label{apu}u_{n,m_n}^\dl(x):=P_mu(x)=\sum_{j=1}^{2^{m}}\zeta_j^{(m_n,\dl)}\Phi_j(x),\ee
$P_m:L^2[0,1]\to L_m$,
\be\label{Lm}L_m=\text{span}\{\Phi_1,\Phi_2,\hdots, \Phi_{2^m}\},\ee
where $\{\Phi_i\}$ are the Haar basis functions (see \cite{IMSB69}):
$\Phi_1(x)=1$ $\forall x\in[0,1]$, and for $j=2^{l-1}+p,\
l=1,2,\hdots,m,\ p=1,2,\hdots,2^{l-1}$
\be\begin{split} \\
\Phi_j(x)&=\left\{
            \begin{array}{ll}
              2^{(l-1)/2}, & \hbox{$x\in [\frac{p-1}{2^{l-1}},\frac{p-1/2}{2^{l-1}})$;} \\
              -2^{(l-1)/2}, & \hbox{$x\in [\frac{p-1/2}{2^{l-1}},\frac{p}{2^{l-1}})$;} \\
              0, & \hbox{otherwise.}
            \end{array}
          \right. \
  \end{split}\ee

Let us formulate an algorithm for obtaining the approximate solution
to \eqref{ILT} using iterative scheme \eqref{it3}, where the
discrepancy-type principle for DSM defined in Section 3 is used as
the stopping rule.
\begin{itemize}
\item[(1)] Given data: $K$, $f_\dl$, $\dl$;
\item[(2)] initialization : $\alpha_0>0$, $\eta\geq 10$, $q\in(0,1)$, $C>2$, $u_{0,m_0}^\dl=0$, $G_0=0$, $n=1$;
\item[(3)] iterate, starting with $n=1$, and stop until the condition
\eqref{stopit} below holds,
\begin{itemize}
\item[(a)]$ a_n=\alpha_0q^n$,
\item[(b)]choose $m_n=\max\left\{\left\lceil \frac{\ln (2c_1/a_n)}{4\ln
2}\right\rceil, \left\lceil\frac{\ln(17/(180\eta
a_n^2))}{2\ln 2}\right\rceil,\left\lceil \frac{\ln
(2c_1/\sqrt{a_n})}{2\ln 2}\right\rceil\right\}$, where $c_1$
is defined in \eqref{cS}, and $a_n$ are defined in (a),
\item[(c)] construct the vectors $v^\dl$ and $g^\dl$:
\be v^\dl_i:=\la K_{m_n}^*f_\dl,\Phi_i\ra,\quad
i=1,2,\hdots,2^{m_n},\ee \be g^\dl_i=\la f_\dl,\Phi_i\ra
\quad i=1,2,\hdots,2^{m_n}, \ee
\item[(d)] construct the matrices $A_{m_n}$ and $B_{m_n}$:
\be\begin{split}
(A_{m_n})_{ij}:&=\sum_{l=1}^{2^{m_n}+1}\beta_l^{(m_n)}\la
k(s_l,\cdot),\Phi_i\ra \la k(s_l,\cdot)\Phi_j\ra,\\
&i,j=1,2,3,\hdots, 2^{m_n},\end{split}\ee\be\begin{split}
(B_{m_n})_{ij}:&=\sum_{l=1}^{2^{m_n}+1}\eta_l^{(m_n)}\la
k(\cdot,s_l),\Phi_i\ra \la k(\cdot,s_l)\Phi_j\ra,\\
&i,j=1,2,3,\hdots, 2^{m_n},\end{split}\ee where
$\beta_i^{(m_n)}$ and $\eta_l^{(m_n)}$ are the quadrature
weights and $s_l$ are the collocation points,
\item[(e)] solve the following two linear algebraic
systems:\be\label{LAS1}(a_nI+A_{m_n})\zeta^{(m_n,\dl)}=v^\dl,\ee
where $(\zeta^{(m_n,\dl)})_i=\zeta_i^{(m_n,\dl)}$ and
\be\label{LAS2}(a_nI+B_{m_n})\gamma^{(m_n,\dl)}=g^{\dl},\ee
where $(\gamma^{(m_n,\dl)})_i=\gamma^{(m_n,\dl)}_i$,
\item[(f)]update the coefficient $\la \zeta^{(m_n,\dl)},\Phi_i\ra$ of the approximate solution $u_{n,m_n}(x)$ in \eqref{apu} by the
iterative formula:
\be\label{uap}u_{n,m_n}^\dl(x)=qu_{n-1,m_{n-1}}^\dl(x)+(1-q)\sum_{j=1}^{2^{m_n}}\zeta^{(m_n,\dl)}_j
\Phi_j(x) ,\ee where \be u_{0,m_0}^\dl(x)=0, \ee
\end{itemize}
until
\be\label{stopit}G_{n,m_n}=qG_{n-1,m_{n-1}}+a_n\|\gamma^{(m_n,\dl)}\|\leq
C\dl^\varepsilon.\ee
\end{itemize} Since $K$ is a selfadjoint operator, the matrix $B_{m_n}$ in step (d) is equal to the matrix $A_{m_n}$. We measure the accuracy of the approximate solution $u_{m_{n_\dl}}^\dl$ by
the following average error formula: \be
Avg:=\frac{\sum_{j=1}^{100}|u(t_j)-u^\dl_{m_{n_\dl}}(t_j)|}{100},\quad
t_1=0,\quad t_j=0.01j,\ j=2,3,\hdots,99, \ee  where $u(t)$ is the
exact solution to problem \eqref{ILT}. In all the experiments we use
$\alpha_0=1$, $q=0.25$, $C=2.01$ and $\varepsilon=0.99$. The linear
algebraic systems \eqref{LAS1} and \eqref{LAS2} are solved using
MATLAB. The levels of noise: $5\%$, $1\%$, and $.05\%$ are used in
the experiments. For the level of noise $5\%$ the stopping condition
is satisfied at $m_{n_\dl}=2$. The resulting average error is
$0.1095.$ When the noise level $\dl$ is decreased to the level of
noise $1\%$, we get the average error $Avg=0.0513$, so the accuracy
of the approximate solution is improved. The parameter $m_{n_\dl}$
for this level of noise is $3$, so one needs to solve a larger
linear algebraic system to get such accuracy. When the noise is
$.5\%$ the average error is improved without increasing the value of
the parameter $m_n$. In this level of noise we get $Avg=0.0452$. The
value of the parameter $m_n$ increases to $4$ as the level of noise
$\dl$ decreases to $0.05\%$. The average error is improved to
$0.0250$. Figure 1 shows the reconstructions with the proposed
iterative scheme for the noise levels: $5\%$, $1\%$, $0.5\%$ and
$0.05\%$.

\begin{figure} \centering
\begin{tabular}{cc}
\epsfig{file=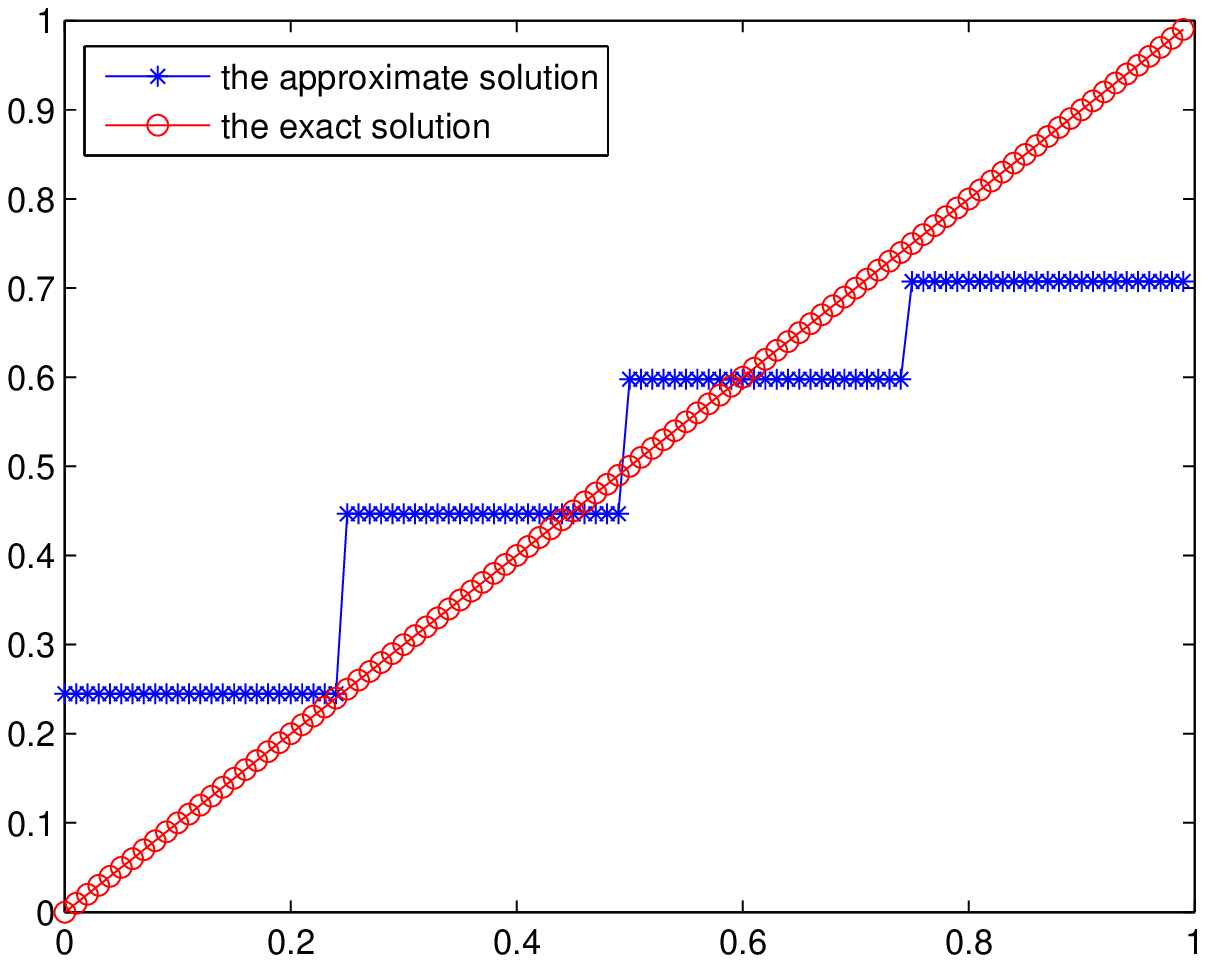,width=5cm,clip=} &
\epsfig{file=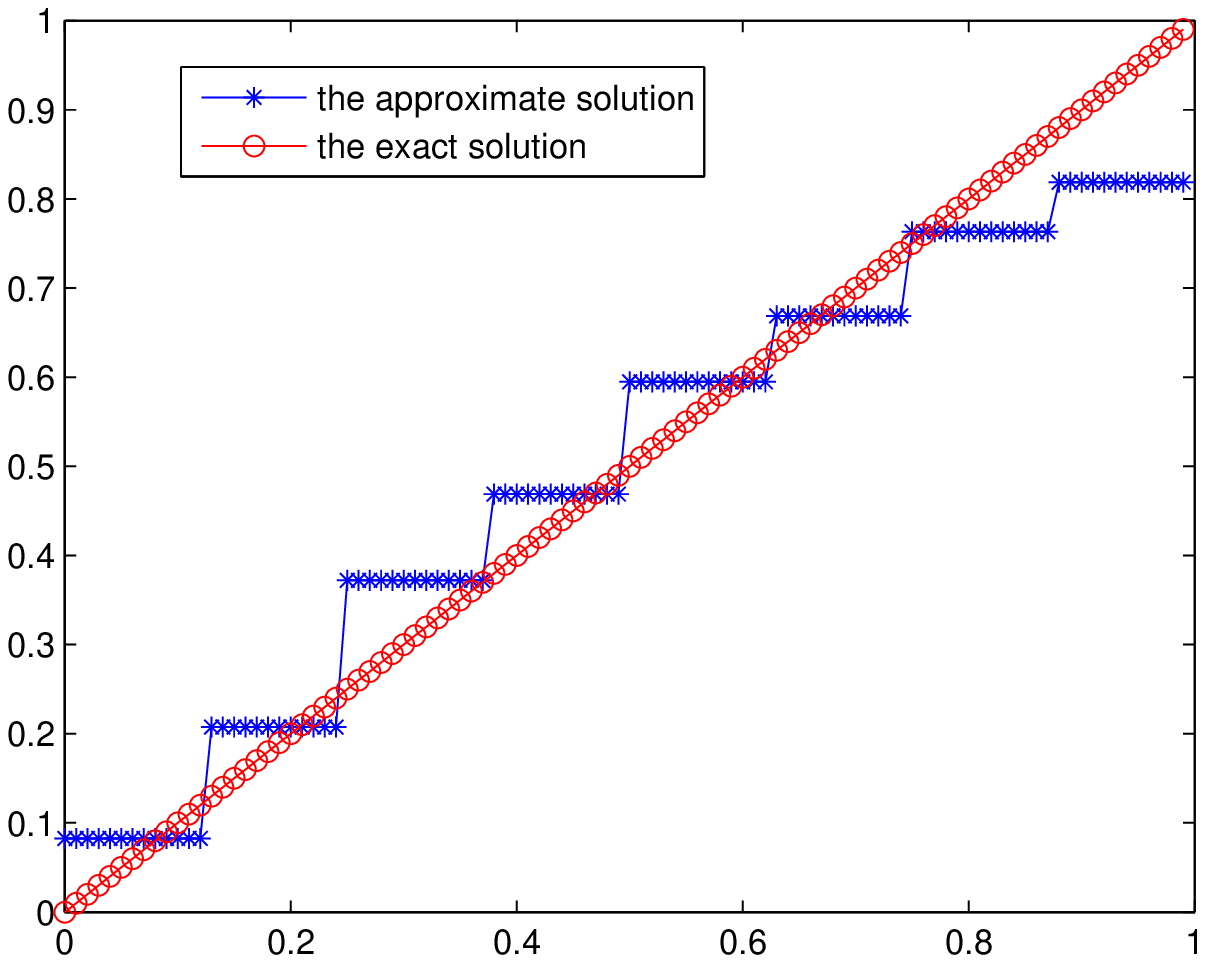,width=5cm,clip=}\\
$\dl=5\%$, $m_{n_\dl}=2$& $\dl=1\%$, $m_{n_\dl}=3$\\
\epsfig{file=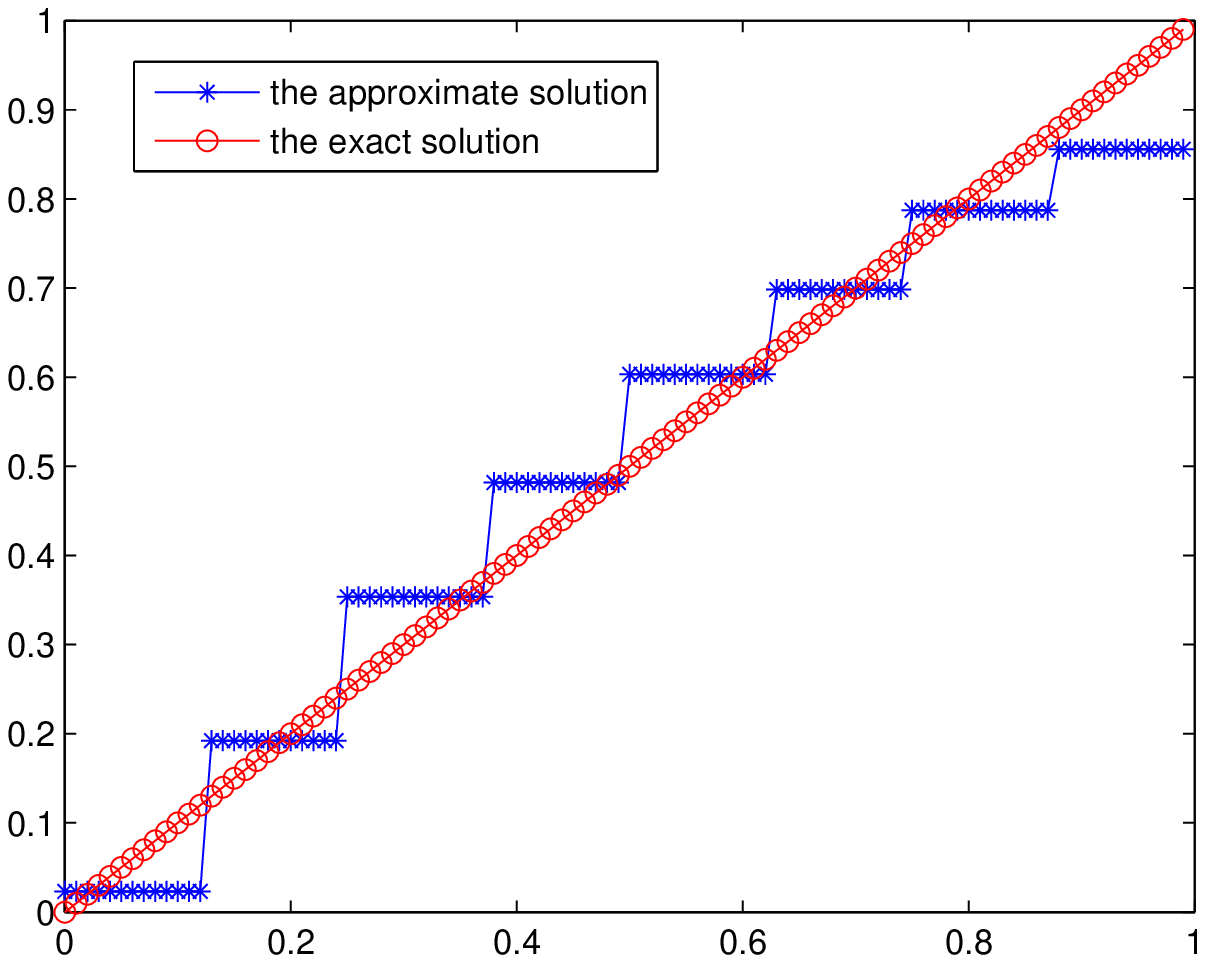,width=5cm,clip=}& \epsfig{file=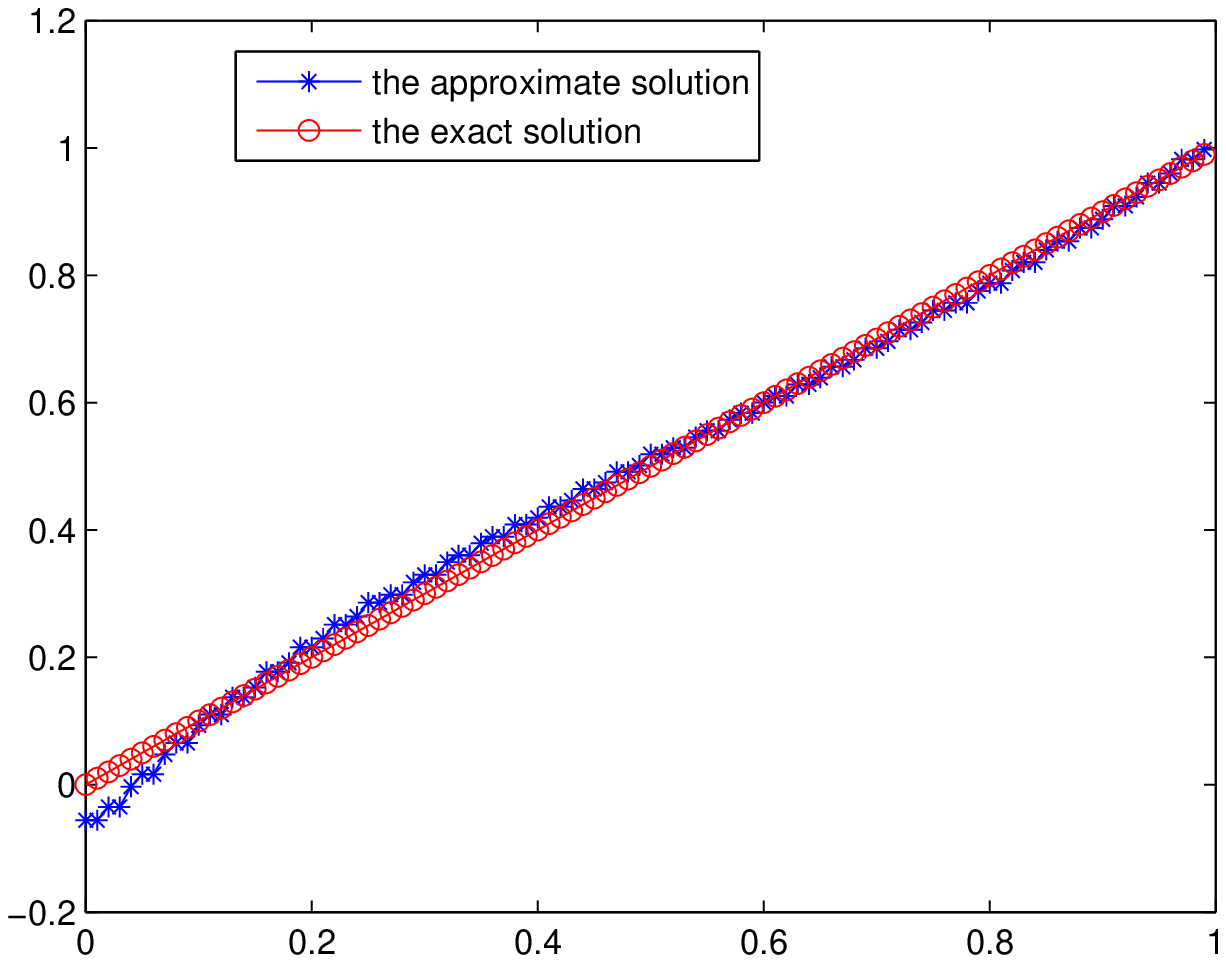,width=5cm,clip=}\\
$\dl=.5\%$, $m_{n_\dl}=3$& $\dl=.05\%$, $m_{n_\dl}=5$\\
\end{tabular}
\caption{Reconstruction of the exact solution $u(t)=t$ using the proposed iterative scheme }\label{fig:fig1}
\end{figure}

We compare the results of the proposed iterative scheme with the
iterative scheme proposed in \cite{SWIAGR09}: \be\label{fit}
u_n^\dl=qu_{n-1}^\dl+(1-q)T_{a_n}^{-1}K^*f_\dl,\quad u_0=0,\quad
a_n=\alpha_0q^n,\ \alpha_0>0.\ee In this iterative scheme we need to
solve the following equation: \be\label{LSf} (a_nI+A)z=A^*f_\dl, \ee
where \be\label{Af}
(A)_{i,j}:=\int_0^1\Phi_i(s)\int_0^1e^{-st}\Phi_j(t)dtds,\quad
i,j=1,2,\hdots,2^m, \ee \be\label{ff} (f_\dl)_i:=\int_0^1
f_\dl(s)\Phi_i(s)ds,\quad i=1,2,\hdots,2^m,\ee and $\Phi_i(x)$ are
the Haar basis functions. In all the experiments the value of the
parameter $m$ in \eqref{Af} and \eqref{ff} is $4$, so the size of
the matrix $A$ in \eqref{LSf} is fixed to $16\times 16$ at each
iteration. The reconstructions obtained by iterative solution
\eqref{fit} are shown in Figure 2.
\begin{figure} \centering
\begin{tabular}{cc}
\epsfig{file=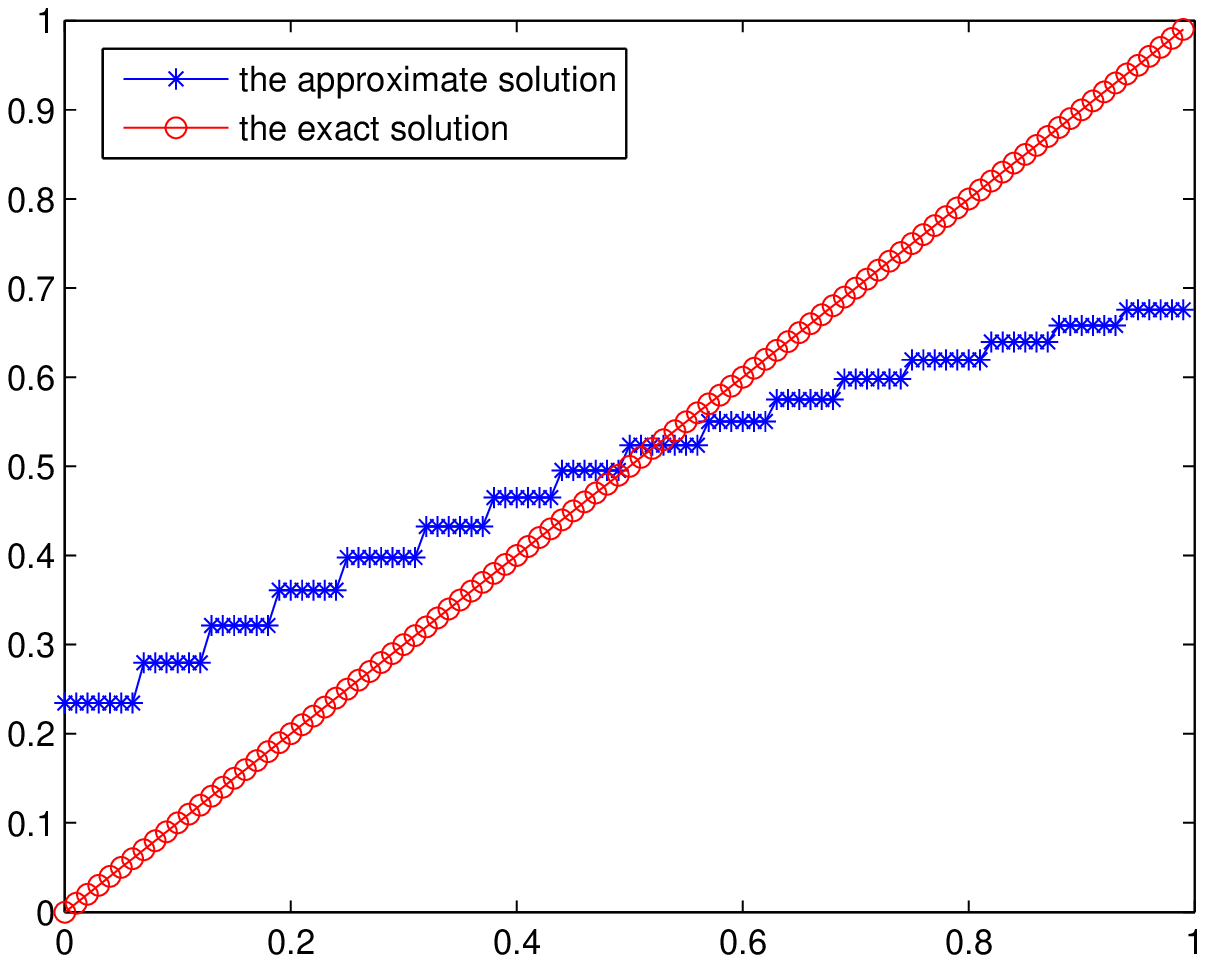,width=5cm,clip=} &
\epsfig{file=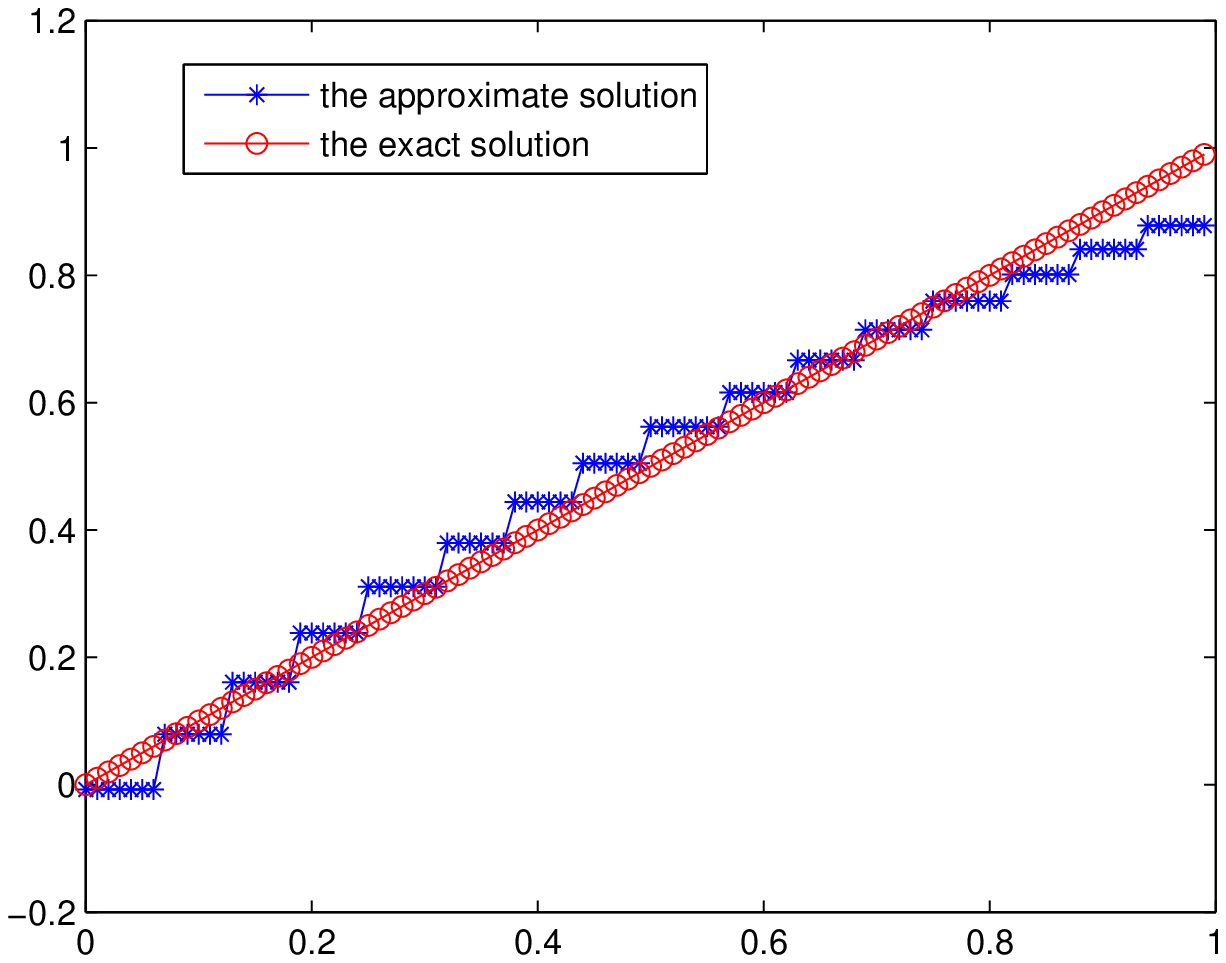,width=5cm,clip=}\\
$\dl=5\%$, $m=4$& $\dl=1\%$, $m=4$\\
\epsfig{file=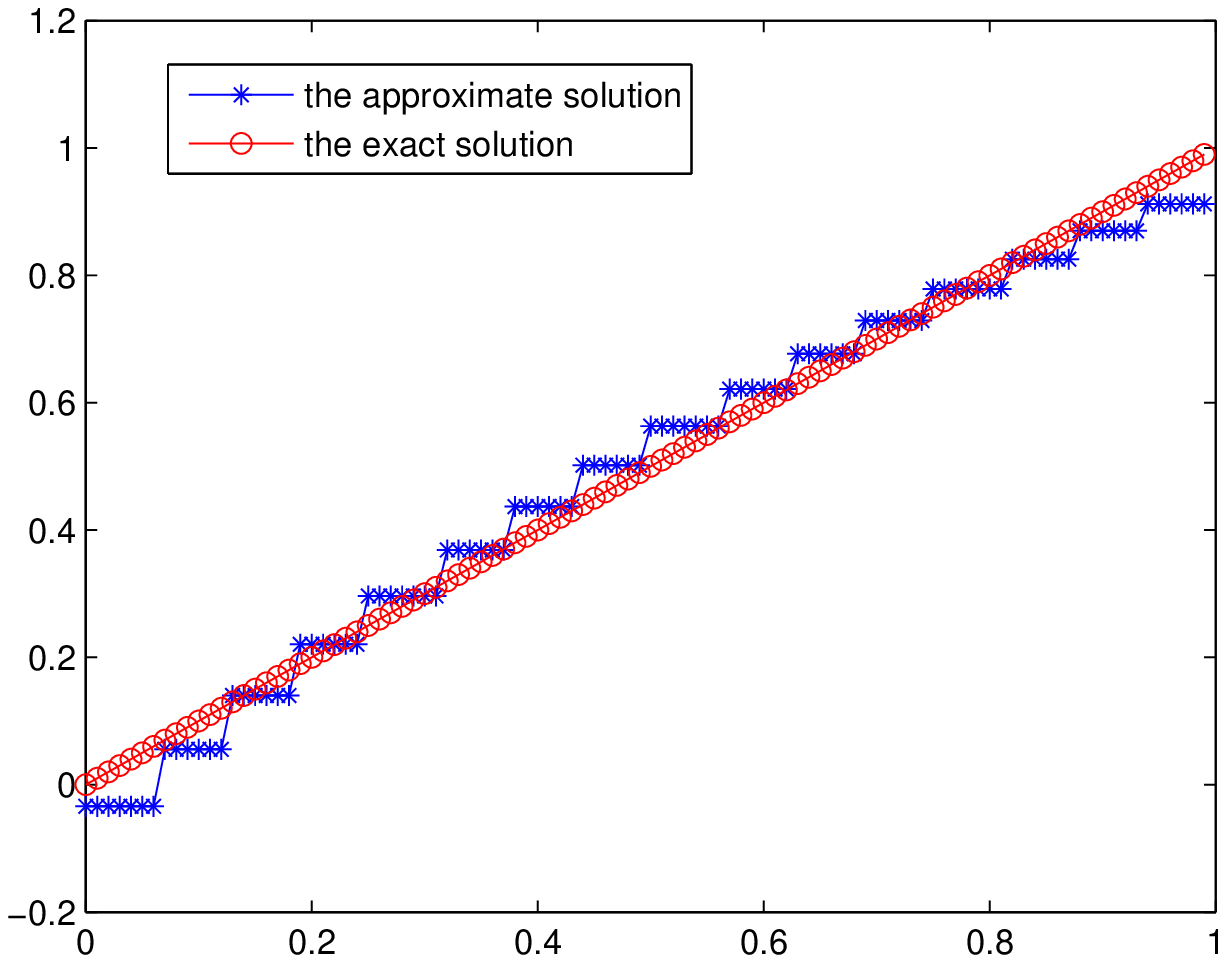,width=5cm,clip=}& \epsfig{file=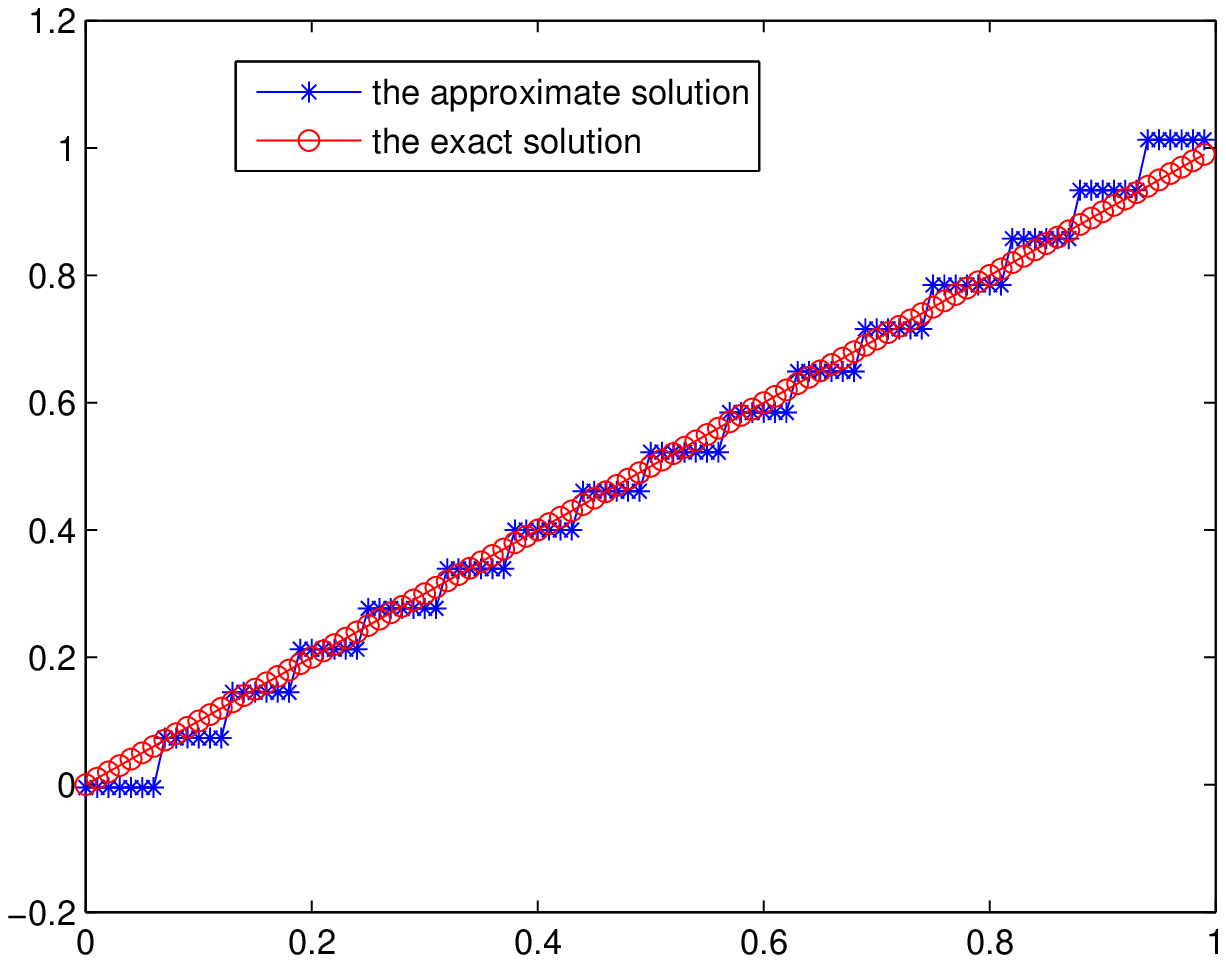,width=5cm,clip=}\\
$\dl=.5\%$, $m=4$& $\dl=.05\%$, $m=4$\\
\end{tabular}
\caption{Reconstruction of the exact solution $u(t)=t$ using iterative scheme \eqref{fit} }\label{fig:fig2}
\end{figure}

In Table 1 we compare the results of the proposed iterative scheme
with of iterative scheme \eqref{fit}. Here the proposed iterative
and iterative scheme \eqref{fit} are denoted by $It_1$ and $It_2$,
respectively. For the levels of noise $5\%,$ $1\%$, $0.5\%$ the CPU
time of iterative scheme \eqref{fit} are larger than of these for
the proposed iterative scheme, since at each iteration of iterative
scheme \eqref{fit} one needs to solve linear algebraic system
\eqref{LSf} with the matrix $A$ of the size $16\times 16$ while in
the proposed iterative scheme one only needs to use smaller sizes of
the matrix $A$ at each iteration. In general the average errors of
the proposed iterative scheme are comparable to of these for
iterative scheme \eqref{fit}.
\begin{table}[htp]\caption{fixed vs adaptive iterative scheme}\label{tab1}
\newcommand{\m}{\hphantom{$-$}}
\renewcommand{\tabcolsep}{.85pc} 
\renewcommand{\arraystretch}{1.2} 
\begin{tabular}{llllllll}
\hline
&\multicolumn{3}{c}{$It_1$}& &\multicolumn{3}{c}{$It_2$}\\
\cline{2-4}\cline{6-8}
$\dl$&$Avg$&$m_{n_\dl}$&CPU time & &$Avg$&$m$&CPUtime\\
& & &(seconds)&&&&(seconds)\\
\hline
$5\%    $&$ 0.1095   $&$ 2   $&$  0.1563  $&&$    0.1346   $&$ 4   $&$ 0.5313  $   \\
$1\%   $&$ 0.0513   $&$ 3   $&$ 0.2188  $&&$    0.0339   $&$ 4   $&$ 0.5313  $   \\
$0.5\%  $&$ 0.0452   $&$ 3  $&$ 0.2344  $&&$    0.0300   $&$ 4  $&$ 0.5469  $   \\
$0.05\% $&$ 0.0250   $&$ 5  $&$ 0.8281  $&&$    0.0206   $&$ 4  $&$  0.5313  $   \\
\hline
\end{tabular}
\end{table}

\section{Conclusion}
A stopping rule with the parameters $m_n$ depending on the
regularization parameters $a_n$ is proposed. The $m_n$ is an
increasing sequence of the regularization parameter $a_n$. This
allows one to start by solving a small size linear algebraic system
\eqref{LAS1}, and one increases the size of the linear algebraic
systems only if $G_n>C\dl^\varepsilon.$ In the numerical example it
is demonstrated that a simple quadrature method, compound Simpson's
quadrature, can be used for approximating the kernel $g(x,z)$,
defined in \eqref{gxz}. Our method yields convergence of the
approximate solution $u_{n,m_\dl}^\dl$ to the minimal norm solution
of \eqref{11}. Numerical experiments show that all the average
errors of the proposed method are comparable to of these for
iterative scheme \eqref{fit}. Our numerical experiments demonstrate
that the adaptive choice of the parameter $m_n$ is more efficient,
in the following sense: the value of the parameters $m_n$ of the
proposed iterative scheme at the noise levels $5\%$,$1\%$ and
$0.5\%$ are smaller than of the parameter $m$, used in  the
iterative scheme \eqref{fit}. Therefore the computational time of
the proposed method at these levels of noise is smaller than the
computational time for the iterative scheme \eqref{fit}. The
adaptive choice of the parameters $m_n$ may give a large size of the
matrix $A_{m_n}$ in \eqref{LAS1}, since $m_n$ is a non-decreasing
sequence depending on the geometric sequence $a_n$,
so the CPU time increases
as the value of the parameter $m_n$ increases. In the iterative
scheme \eqref{fit} the size of the matrix $A$ in \eqref{LSf} is
fixed at each iteration, so the CPU time depends on the number of
iterations. The drawback of using a fixed size $2^m\times 2^m$ of
the matrix $A$ in \eqref{LSf} at each iteration is: the solution
$u_n^\dl$, defined by formula \eqref{fit}, where $n=n(\dl)$ is found
by the stopping rule \eqref{srule} with $m_n=m$ $\forall n$, may
approximate the minimal norm solution on the finite-dimensional
space $L_m=\text{span}\{\Phi_1,\Phi_2,\hdots,\Phi_{2^m}\}$ not
accurately, so that for some levels of the noise the exact solution
to problem \eqref{ILT} will not be well approximated by any function
from $L_m$. From Table 1 one can see that the number of basis functions
used for an approximation of the minimal norm solution with the accuracy
0.1095 by the iterative scheme with the adaptive choice of $m_n$ is four
times smaller than the number of these functions used in the iterative
scheme with a fixed $m$, while the accuracy is 0.1095 in $It_1$ and
$0.1346$ in $It_2$ (see line 1 in Table 1).


\begin{thebibliography}{00}
\bibitem{PDPR84} P.J. Davis and P. Rabinowitz, \textit{Methods of numerical
integration}, Academic Press, INC., London, 1984.
\bibitem{SWIAGR09} S.W. Indratno and A.G. Ramm, Dynamical Systems
Method for solving ill-conditioned linear algebraic systems, Int.
Journal of Computing Science and Mathematics, (to appear).
\bibitem{VIVT02} V. Ivanov,
V. Vasin, V. Tanana, \textit{Theory of linear inverse and ill-posed
problems and its applications}, VSP, Utrecht, 2002.
\bibitem{MRZ84} V.
Morozov, \textit{Methods of solving incorrectly posed problems},
Springer Verlag, New York, 1984.
\bibitem{RAMM05}
A.~G.~Ramm, \textit{Inverse problems}, Springer, New York, 2005.
\bibitem{RAMM499}
A.~G.~Ramm, \textit{Dynamical systems method for solving operator equations},
Elsevier, Amsterdam, 2007.
\bibitem{RAMM525a}
A. G. Ramm, Discrepancy principle for DSM, I, Comm. Nonlin. Sci. and
Numer. Simulation, 10, N1, 95-101, 2005.
\bibitem{RAMM525b}
A. G. Ramm, Discrepancy principle for DSM II, Comm. Nonlin. Sci. and
Numer. Simulation, 13, 1256-1263, 2008.

\bibitem{IMSB69} I. M. Sobol, \textit{Multidimensional Quadrature Formulas
and Haar Functions}, Nauka, Moscow, 1969.
\end{thebibliography}
\end{document}